\documentclass[12pt, reqno]{amsart}
\usepackage{ifthen,amsfonts,
amsmath,amssymb,
epsfig,color,graphicx}

\makeatletter

\def\fnum{equation}
\def\fnum{theorem}

\begingroup

\newtheorem{theorem}{Theorem}[section]
\newtheorem{lemma}[theorem]{Lemma}
\newtheorem{propos}[theorem]{Proposition}

\endgroup

\newtheorem{definition}[theorem]{Definition}

\newtheorem{remark}[theorem]{Remark}

\newtheorem{Rem}[\fnum]{Remark}

\newtheorem*{Cla}{Claim}

\numberwithin{equation}{section}

\newcommand{\per}{{\rm Per}\,}

\newcommand{\an}{{\rm An}}
\newcommand{\eps}{{\varepsilon}}

\newcommand{\diam}{{\text {diam}}}
\newcommand{\dist}{{\text {dist}}}

\newcommand{\Sing}{{\rm Sing}}

\newcommand{\R}{{\rm R}}

\def\RR{{\mathbb{R}}}

\def\NN{{\mathbb{N}}}

\newcommand{\cB}{{\mathcal{B}}}
\newcommand{\cC}{{\mathcal{C}}}
\newcommand{\cD}{{\mathcal{D}}}

\newcommand{\cF}{{\mathcal{F}}}

\newcommand{\cH}{{\mathcal{H}}}

\newcommand{\cO}{{\mathcal{O}}}

\newcommand{\cV}{{\mathcal{V}}}

\newcommand\B{{\mathcal{B}}}

\renewcommand\d{{\rm d}\,}

\def\ov#1{\overline{#1}}

\newcommand\haus{\mathcal{H}}

\newcommand\N{{\mathbb N}}

\newcommand\res{\mathop{\hbox{\vrule height 7pt width .5pt depth 0pt
\vrule height .5pt width 6pt depth 0pt}}\nolimits}
\newcommand\supp{{\rm supp}\,}  
\newcommand\Inj{{\textrm{Inj}\,}}

\newcommand\An{{\mathcal{AN}}}

\newcommand\ind{{\bf 1}}

\begin{document}

\title[The existence of embedded minimal hypersurfaces]
{The existence of embedded minimal hypersurfaces}

\author{Camillo De Lellis}%
\author{Dominik Tasnady}%
\address{Zurich university}

\email{camillo.delellis@math.unizh.ch and 
dominik.tasnady@math.unizh.ch}

\begin{abstract}
We give a shorter proof of the existence of nontrivial
closed minimal hypersurfaces in closed smooth 
$(n+1)$--dimensional Riemannian manifolds,
a theorem proved first by Pitts for $2\leq n\leq 5$
and extended later by Schoen and Simon to any $n$. 
\end{abstract}

\maketitle
 
\section{Introduction}
In this paper we give a proof of the following theorem,
a natural generalization of the classical existence of 
nontrivial simple closed geodesics in closed $2$--d Riemannian 
manifolds.

\begin{theorem}\label{t:existence}
Let $M$ be an $(n+1)$-dimensional smooth closed Riemannian manifold. Then there is a nontrivial 
embedded minimal hypersurface $\Sigma\subset M$ 
without boundary with a singular set $\Sing\, \Sigma$
of Hausdorff dimension at most $n-7$.
\end{theorem}

More precisely, $\Sigma$ is 
a closed set of finite $\cH^n$--measure and $\Sing\, \Sigma\subset
\Sigma$ is the smallest closed set $S$ such that
$M\setminus S$ is a smooth embedded hypersurface
($\Sigma\setminus \Sing\, \Sigma$ is in fact
analytic if $M$ is analytic). In this paper \textit{smooth} will always mean $C^\infty$. In fact, the result remains true for any $C^4$ Riemannian manifold $M$, $\Sigma$ then will be of class $C^2$ (see \cite{SS}). Moreover
$\int_{\Sigma\setminus\Sing\, \Sigma} \omega = 0$
for any exact $n$--form on $M$. 
The case $2\leq n\leq 5$ was proved by 
Pitts in his groundbreaking monograph
\cite{P}, an outstanding contribution which
triggered all the subsequent research in the
topic. The general case was proved
by Schoen and Simon in \cite{SS}, 
building heavily upon the work of Pitts. 

The monograph \cite{P} can be ideally split into two parts. 
The first half of the book 
implements a complicated existence
theory for suitable ``weak generalizations'' of 
global minimal submanifolds, which is
a version of the classical min--max argument introduced by Birkhoff
for $n=1$ (see \cite{B}).
The second part contains the regularity 
theory needed to prove Theorem \ref{t:existence}. The curvature
estimates of \cite{SSY} for stable minimal surfaces are
a key ingredient of this part: the core
contribution of \cite{SS} is the extension 
of these fundamental estimates 
to any dimension,
which enabled the authors to complete Pitts' program for $n> 5$. 

\cite{SS} gives also a quite readable account of parts of 
Pitts' regularity theory.
To our knowledge, there is instead no contribution to clarify
other portions of the monograph, at least in general dimension.
Indeed, for $n=2$, the unpublished PhD thesis
of Smith (see \cite{Sm}) gives a powerful variant of Pitts 
approach. Building on ideas of Simon, 
the author proved the existence 
of minimal embedded $2$--spheres in any $M$ which is topologically 
a $3$--sphere (further theorems
in general Riemannian $3$--manifolds have been claimed in \cite{PR1}; 
\cite{CD} and \cite{DP} contain a complete proof of the Simon--Smith 
Theorem and of a statement 
in the direction of \cite{PR1}).
Smith's aproach relies heavily on the features of  
$2$--dimensional surfaces in $3$--manifolds, most notably on the celebrated 
paper \cite{MSY}, and therefore it is not 
feasible in higher dimensions.

This paper gives a much simpler proof of Theorem \ref{t:existence}. 
Our contribution draws heavily on the existing
literature and follows Pitts in many aspects.
However we introduce some new ideas which, in spite of 
their simplicity, allow us to shorten
the proof dramatically. 
These contributions are contained in Sections 3 and 4 of the paper, 
but we prefer to give
a complete account of the proof of Theorem \ref{t:existence}, containing
all the necessary technical details.
We leave aside only those facts which are either
(by now) classical results or for which we can give
a precise reference. 

\subsection{Min--max surfaces}
In what follows $M$ will denote an $(n+1)$-dimensional smooth Riemannian 
manifold without boundary. First of all we need to generalize slightly
the standard notion of $1$-parameter family of hypersurfaces, allowing
for some singularities.

\begin{definition}\label{d:sweep}
A family $\{\Gamma_t\}_{t\in [0,1]^k}$ of closed subsets of $M$ 
with finite $\cH^n$--measure is called a {\em generalized smooth family} 
if 
\begin{itemize}
\item[(s1)] For each $t$ there is a finite $P_t\subset M$ such that 
$\Gamma_t$ is a smooth hypersurface in $U$;
\item[(s2)] $\cH^n(\Gamma_t)$ depends smoothly on $t$ and
$t\mapsto \Gamma_t$ is continuous in the Hausdorff sense;
\item[(s3)] on any $U\subset\subset M\setminus P_{t_0}$,  
$\Gamma_t \stackrel{t\rightarrow t_0}{\longrightarrow} \Gamma_{t_0}$ 
smoothly in $U$. 
\end{itemize}
$\{\Gamma_t\}_{t\in [0,1]}$ is a {\em sweepout} of
$M$ if there exists a family $\{\Omega_t\}_{t\in [0,1]}$ 
of open sets such that
\begin{itemize}
\item[(sw1)] $(\Gamma_t \setminus \partial \Omega_t) \subset P_t$ for any $t$;
\item[(sw2)] $\Omega_0=\emptyset$ and $\Omega_1 = M$; 
\item[(sw3)] ${\rm Vol} (\Omega_t\setminus\Omega_s)
+{\rm Vol} (\Omega_s\setminus\Omega_t) \to 0$ 
as $t\to s$.
\end{itemize}
\end{definition}

\begin{Rem}\label{r:smooth_conv}
The convergence in (s3) means, as usual, that, if 
$U\subset\subset M\setminus
P_{t_0}$, then there is $\delta>0$ such that, for $|t-t_0|<\delta$,
$\Gamma_t\cap U$ is the graph of a function $g_t$ over 
$\Gamma_{t_0}\cap U$.
Moreover, given $k\in \N$ and $\eps>0$,
$\|g_t\|_{C^k}< \eps$ provided $\delta$ is sufficiently small.
\end{Rem}

We introduce the singularities $P_t$ for two important reasons.
They allow for the change of topology which, for $n> 2$,
is a fundamental tool of the regularity theory.
It is easy to exhibt sweepouts as in Definition \ref{d:sweep}
as it is witnessed by the
following proposition.  

\begin{propos}\label{p:Morse}
Let $f: M \to [0,1]$ be a smooth Morse function. 
Then $\{\{f=t\}\}_{t\in [0,1]}$ is a sweepout.
\end{propos}

The obvious proof is left to the reader.
For any generalized family $\{\Gamma_t\}$ we set
\begin{equation}
 \cF(\{\Gamma_t\})\;:=\;\max_{t\in[0,1]} \cH^n(\Gamma_t).
\end{equation}
A key property of sweepouts is an obvious 
consequence of the isoperimetric inequality.

\begin{propos}\label{p:isop} There exists
$C(M)>0$ such that $\cF (\{\Gamma_t\})
\geq C (M)$ for every sweepout. 
\end{propos}
\begin{proof} Let $\{\Omega_t\}$ be as in Definition \ref{d:sweep}.
Then, there is $t_0\in [0,1]$ such that ${\rm Vol} (\Omega_{t_0})
= {\rm Vol} (M)/2$. We then conclude 
$\cH^n (\Gamma_{t_0})\;\geq\; c_0^{-1}  
(2^{-1}{\rm Vol}\, (M))^{\frac{n}{n+1}}$, where $c_0$ is the
isoperimetric constant of $M$.
\end{proof}

For any family $\Lambda$ of sweepouts we define
\begin{equation}
 m_0(\Lambda)\;:=\;\inf_\Lambda \cF\;=\;
\inf_{\{\Gamma_t\}\in\Lambda}\left[\max_{t\in[0,1]} 
\cH^n(\Gamma_t)\right].
\end{equation}
By Proposition \ref{p:isop}, $m_0 (\Lambda)\geq C(M)>0$.
A sequence $\{\{\Gamma_t\}^k\}\subset\Lambda$ is \textit{minimizing} if 
$$
\lim_{k\to\infty}\cF(\{\Gamma_t\}^k)\;=\;m_0(\Lambda)\, .
$$ 
A sequence of surfaces $\{\Gamma_{t_k}^k\}$ is a 
\textit{min-max sequence} 
if $\{\{\Gamma_t\}^k\}$ is minimizing and  $\cH^n(\Gamma_{t_k}^k)
\to m_0(\Lambda)$. 
The min--max construction is applied 
to families of sweepouts which are closed under
a very natural notion of homotopy.

\begin{definition}\label{d:homotopy}
Two sweepouts $\{\Gamma^0_s\}$ and $\{\Gamma^1_s\}$ are homotopic
if there is a generalized family $\{\Gamma_t\}_{t\in [0,1]^2}$
such that $\Gamma_{(0,s)} = \Gamma^0_s$ and $\Gamma_{(1,s)}=\Gamma^1_s$.
A family $\Lambda$ of sweepouts is called {\em homotopically closed} if 
it contains the homotopy class of each of its elements.
\end{definition}

Ultimately, this paper gives a proof of the following Theorem,
which, together with Proposition \ref{p:Morse}, implies
Theorem \ref{t:existence} for $n\geq 2$ (recall that Morse functions
exist on every smooth compact Riemannian manifold without boundary;
see Corollary 6.7 of \cite{Mi}).

\begin{theorem}\label{t:main} Let $n\geq 2$.
For any homotopically closed family 
$\Lambda$ of sweepouts there is a min--max sequence
$\{\Gamma_{t_k}^k\}$ converging (in the sense of varifolds)
to an embedded minimal hypersurface $\Sigma$ as in Theorem \ref{t:existence}.
Multiplicity is allowed.
\end{theorem}

The smoothness assumption on the metric $g$
can be relaxed easily to $C^4$. The ingredients of the
proof where this regularity is needed are:
the regularity theory for
the Plateau problem, the unique continuation for classical
minimal surfaces and the Schoen--Simon compactness theorem.
$C^4$ suffices for all of them.
   

The paper is organized as follows: Section \ref{s:prel} contains
some preliminaries, Section \ref{s:overview} gives an overview
of the proof of Theorem \ref{t:main}, Section \ref{s:am} contains
the existence theory and the Sections \ref{s:rep} and \ref{s:reg} contain
the regularity theory.

\section{Preliminaries}\label{s:prel}

\subsection{Notation}\label{ss:notat}
Throughout this paper our notation will be consistent with
the one introduced in Section 2 of \cite{CD}. We summarize it
in the following table.

\medskip

\begin{tabular}{lll}
$\Inj(M)$ && the injectivity radius of $M$;\\
$B_\rho (x)$, $\ov{B}_\rho (x)$, $\partial B_\rho (x)$ && the open 
and closed ball, the distance sphere in $M$;\\
$\diam (G)$ && the diameter of $G\subset M$;\\
$d(G_1, G_2)$&& $\inf_{x\in G_1, y\in G_2} d (x,y)$\, ;\\
$\B_\rho$ && the ball of radius $\rho$ and centered in $0$ in $\RR^n$;\\
$\exp_x$ && the exponential map in $M$ at $x\in M$;\\
$\an (x,\tau, t)$ && the open annulus 
$B_t (x)\setminus \overline{B}_\tau (x)$;\\
$\An_r (x)$ && the set 
$\{\an (x, \tau, t) \mbox{ where $0<\tau<t<r$}\}$;\\
$\mathcal{X} (M)$, $\mathcal{X}_c (U)$ && smooth vector fields, 
smooth vector fields supported in $U$. 
\end{tabular}

\begin{remark}
In \cite{CD} the authors erroneously define $d$ as the Hausdorff distance. However, for the
purposes of both this and that paper, the correct definition of $d$ is the one
given here, since in both cases the following fact plays a fundamental role: $d (A,B)>0
\Longrightarrow A\cap B=\emptyset$.
Note that, unlike the Hausdorff distance, $d$ is not a distance
on the space of compact sets.   
\end{remark}

\subsection{Caccioppoli sets and Plateau's problem}\label{ss:caccioppoli}
We give here a brief account of the theory
of Caccioppoli sets. A standard reference is
\cite{Giu}. Let $E\subset M$ be a measurable set
and consider its indicator function $\ind_E$ (taking
the value $1$ on $E$ and $0$ on $M\setminus E$). The perimeter
of $E$ is defined as
$$
\per (E):=\sup \left\{\int_M \ind_E\,{\rm div}\, \omega\,:\,
\omega\in \mathcal{X} (M), \, \|\omega\|_{C^0}\leq 1\right\}.
$$
A Caccioppoli set is a set $E$ for which $\per (E)<\infty$.
In this case the distributional derivative $D \ind_E$ is
a Radon measure and $\per E$ corresponds to its total variation.
As usual, the perimeter of $E$ in an open set $U$,
denoted by $\per (E, U)$, is the total variation of $D\ind_E$ in
the set $U$.

We follow De Giorgi and, given a Caccioppoli
set $\Omega\subset M$ and an open
set $U\subset M$, we consider
the class 
\begin{equation}\label{e:CUOmega}
\mathcal{P} (U, \Omega)\;:=\; \{ \Omega'\subset M \mbox{}\;:\;
\Omega'\setminus U = \Omega\setminus U\}\, .
\end{equation}
The theorem below states the fundamental existence
and interior regularity 
theory for De Giorgi's solution of the Plateau problem, which
summarizes results of De Giorgi, Almgren, Simons and Federer
(see \cite{Giu} for the case $M=\RR^{n+1}$ and Section 37 of
\cite{Si} for the general case).

\begin{theorem}\label{t:DeGiorgi} 
Let $U, \Omega\subset M$ be, respectively,
an open and a Caccioppoli set. Then there exists
a Caccioppoli set $\Xi\in \mathcal{P} (U, \Omega)$ minimizing
the perimeter. Moreover, any such minimizer is, in $U$,
an open set whose boundary is smooth outside of a singular
set of Hausdorff dimension at most $n-7$.
\end{theorem}

\subsection{Theory of varifolds}
We recall here some basic facts from the 
theory of varifolds; see for instance  
chapters 4 and 8 of \cite{Si} for further information.
Varifolds are a convenient way of 
generalizing surfaces to a category that 
has good compactness properties. An advantage of varifolds, over other
generalizations (like currents), is that they do not allow for
cancellation of mass. This last property is fundamental for the
min--max construction.
If $U$ is an open 
subset of $M$, any finite nonnegative measure on the Grassmannian 
$G (U)$ of unoriented $n$--planes on 
$U$ is said to be an {\em $n$--varifold in $U$}. 
The space of $n$--varifolds is denoted by $\mathcal{V} (U)$  
and we endow it with the topology of
the weak$^*$ convergence in the sense of
measures. Therefore, a sequence $\{V^k\}\subset \mathcal{V}
(U)$ converges to $V$ if 
$$
\lim_{k\to \infty} \int \varphi (x, \pi)\, dV^k (x, \pi)
\;=\; \int \varphi (x, \pi)\, dV (x, \pi)
\qquad \mbox{for every $\varphi\in C_c (G(U))$.}
$$
Here $\pi$ denotes an $n$--plane of $T_x M$.
If $U'\subset U$ and $V\in \mathcal{V} (U)$, then 
$V\res U'$ is the restriction of the measure $V$ to 
$G (U')$. Moreover, $\|V\|$ is the nonnegative measure
on $U$ defined by
$$
\int_U \varphi (x) \,d\|V\| (x)\;=\;
\int_{G(U)} \varphi (x) \,dV (x, \pi)\qquad \forall \varphi\in C_c
(U)\, .
$$
The support of $\| V\|$, denoted by $\supp \,(\|V\|)$, is the 
smallest closed set outside
which $\|V\|$ vanishes identically.
The number $\|V\|(U)$ will be
called the {\em mass of $V$ in $U$}. 

Recall also that an $n$--dimensional rectifiable set is the countable union
of closed subsets of $C^1$ surfaces (modulo sets of $\haus^n$--measure 0).
If $R\subset U$ is an $n$--dimensional rectifiable set 
and $h:R\to \RR_{+}$ is a Borel function, then the {\em
varifold $V$ induced by $R$} is defined by 
\begin{equation}\label{e:defvar}
\int_{G (U)} \varphi (x, \pi) \,dV (x, \pi)=
\int_R h(x) \varphi (x, T_x R) \,d\haus^n (x)\, \quad \forall
\varphi\in C_c (G (U))\, .
\end{equation}
Here $T_x R$ denotes the tangent plane to $R$ in $x$.
If $h$ is integer--valued, then we say that $V$ is an 
{\em integer rectifiable varifold}.
If $\Sigma=\bigcup n_i \Sigma_i$, then 
by slight abuse of notation we use $\Sigma$ for the 
varifold induced by $\Sigma$ via \eqref{e:defvar}.

If $\psi:U\to U'$ is a diffeomorphism and $V\in \mathcal{V} (U)$, 
$\psi_\sharp V\in \mathcal{V} (U')$ is the varifold defined by
$$
\int \varphi(y, \sigma)\, d(\psi_\sharp V) (y, \sigma)
\;=\; \int J \psi (x, \pi)\, \varphi 
(\psi (x), d\psi_x (\pi))\, dV (x, \pi)\, ;
$$
where $J \psi (x, \pi)$ denotes the Jacobian determinant (i.e. the area
element) of the differential $d\psi_x$ restricted to the plane $\pi$;
cf. equation (39.1) of \cite{Si}. Obviously, if $V$ is induced
by a $C^1$ surface $\Sigma$, $V'$ is induced by $\psi (\Sigma)$.

Given $\chi\in\mathcal{X}_c(U)$, let $\psi$ be the isotopy
generated by $\chi$, i.e. ${\textstyle \frac{\partial
    \psi}{\partial t} = \chi( \psi)}$. The   
first and second variation of $V$ with respect to $\chi$ 
are defined as
$$
[\delta V] (\chi) \;=\; \left. \frac{d}{dt} 
(\|\psi (t, \cdot)_\sharp V\|)(U)
\right|_{t=0}\qquad
\mbox{and} \qquad [\delta^2 V] (\chi) \;=\; 
\left. \frac{d^2}{dt^2} 
(\|\psi (t, \cdot)_\sharp V\|)(U)
\right|_{t=0}\, , 
$$
cf. sections 16 and 39 of \cite{Si}. $V$ is said to be
{\em stationary} (resp. {\em stable}) in $U$ if $[\delta V]
(\chi) =0$ (resp. $[\delta^2 V] (\chi)\geq0$)
for every $\chi\in \mathcal{X}_c (U)$. If $V$ is induced
by a surface $\Sigma$ with $\partial \Sigma\subset
\partial U$, $V$ is stationary (resp. stable)
if and only if $\Sigma$ is minimal (resp. stable).

Stationary varifolds in a Riemannian manifold satisfy
the monotonicity formula, i.e. there exists
a constant $\Lambda$ (depending on the ambient manifold 
$M$) such that the function
 \begin{equation}\label{e:MonFor}
f(\rho)\;:=\; e^{\Lambda \rho}\frac{\|V\| (B_\rho (x))}{\omega_n
\rho^n}
\end{equation}
is nondecreasing for every $x$ (see Theorem 17.6 of \cite{Si}; $\Lambda=0$
if the metric of $M$ is flat).
This property allows us to define the
{\em density} of a stationary varifold $V$ at $x$, by
$$
\theta (x, V)\;=\; \lim_{r\to 0} \frac{\|V\| (B_r (x))}{\omega_n r^n}.
$$  

\subsection{Schoen--Simon curvature estimates}\label{ss:SS}
Consider an orientable $U\subset M$. We look here
at closed sets $\Gamma\subset M$ of codimension $1$
satisfying the following regularity assumption:
\begin{itemize}
\item[(SS)] $\Gamma \cap U$ is a smooth
embedded hypersurface outside a closed set $S$
with $\cH^{n-2} (S)=0$.
\end{itemize}
$\Gamma$ induces an integer rectifiable varifold $V$.
Thus $\Gamma$ is said to be minimal
(resp. stable) in $U$ with respect to the metric
$g$ of $U$ if $V$ is stationary (resp. stable).
The following compactness
theorem, a consequence of the Schoen--Simon
curvature estimates (cp. with
Theorem 2 of Section 6 in \cite{SS}),
is a fundamental tool in this note.

\begin{theorem}\label{t:SScomp}
Let $U$ be an orientable open subset of a manifold
and $\{g^k\}$ and $\{\Gamma^k\}$, respectively,
sequences of smooth metrics on $U$ and of  
hypersurfaces $\{\Gamma^k\}$ satisfying (SS). 
Assume that the metrics $g^k$ converge smoothly
to a metric $g$, that each $\Gamma^k$ is 
stable and minimal relative to the metric $g^k$
and that $\sup \cH^n (\Gamma^k)<\infty$. 
Then there are
a subsequence of $\{\Gamma^k\}$ (not relabeled),
a stable stationary varifold $V$ in $U$ (relative to the metric $g$)
and a closed set $S$ of Hausdorff dimension at most $n-7$
such that
\begin{itemize}
\item[(a)] $V$ is a smooth embedded hypersurface in $U\setminus S$;
\item[(b)] $\Gamma^k\to V$ in the sense of varifolds in $U$;
\item[(c)] $\Gamma^k$ converges smoothly to $V$ on every
$U'\subset\subset U\setminus S$.
\end{itemize}
\end{theorem}

\begin{Rem}\label{r:SSconvergence}
The precise meaning of (c) is as follows:
fix an open $U''\subset U'$ where the varifold $V$ is an integer
multiple $N$ of a smooth oriented surface $\Sigma$. Choose
a normal unit vector field on $\Sigma$ (in the metric $g$)
and corresponding 
normal coordinates in a tubular neighborhood.
Then, for $k$ sufficiently large, $\Gamma^k\cap U''$ 
consists of $N$ disjoint smooth surfaces $\Gamma^k_i$ 
which are graphs of functions
$f^k_i\in C^\infty (\Sigma)$ in the chosen coordinates.
Assuming, w.l.o.g.,
$f^k_1\leq f^k_2 \leq \ldots \leq f^k_N$,
each sequence $\{\Gamma^k_i\}_k$
converges to $\Sigma$ in the sense of Remark \ref{r:smooth_conv}.
\end{Rem}

Note the following obvious corollary of
Theorem \ref{t:SScomp}: if $\Gamma$
is a stationary and stable surface satisfying (SS), then
the Hausdorff dimension of $\Sing\, \Gamma$ is, in fact,
at most $n-7$. Since we will deal very often
with this type of surfaces, we will use the following 
notational convention.

\begin{definition}\label{d:convention}
 Unless otherwise specified, a {\em hypersurface}
$\Gamma\subset U$ is a closed set of codimension $1$
such that 
$\overline{\Gamma}\setminus \Gamma \subset \partial U$
and $\Sing\, \Gamma$ has Hausdorff dimension
at most $n-7$. 
The words ``stable'' and ``minimal'' are then used as explained
at the beginning of this subsection. For instance,
the surface $\Sigma$ of Theorem \ref{t:existence} is
a {\em minimal hypersurface}.
\end{definition}

\section{Proof of Theorem \ref{t:main}}
\label{s:overview}

\subsection{Isotopies and stationarity}
It is easy to see that not all
min--max sequences converge to stationary
varifolds (see \cite{CD}). In general, for any minimizing sequence
$\{\{\Gamma_t\}^k\}$ there is at least one min--max sequence converging
to a stationary varifold. For technical reasons, it is useful
to consider minimizing sequences $\{\{\Gamma_t\}^k\}$ with the 
additional property
that {\em any} corresponding min--max sequence converges
to a stationary varifold. The existence 
of such a sequence, which roughly speaking follows from ``pulling
tight'' the surfaces of a minimizing sequence, is an important 
conceptual step and
goes back to Birkhoff in the case of geodesics and 
to the fundamental work of Pitts in the general case 
(see also \cite{CM1} and \cite{CM2} for other applications 
of these ideas). In order to state it, we need
some terminology.

\begin{definition}\label{d:isotopies}
Given a smooth map $F: [0,1]\to \mathcal{X} (M)$,
for any $t\in [0,1]$
we let $\Psi_t: [0,1]\times M\to M$ be the one--parameter
family of diffeomorphisms generated by the vectorfield $F (t)$.
If $\{\Gamma_t\}_{t\in [0,1]}$ is a sweepout,
then $\{\Psi_t (s,\Gamma_t)\}_{(t,s)\in [0,1]^2}$ is
a homotopy between $\{\Gamma_t\}$ and
$\{\Psi_t (1, \Gamma_t)\}$. These will be called
{\em homotopies induced by ambient isotopies}.
\end{definition}

We recall that the weak$^\ast$ 
topology on the space $\mathcal{V} (M)$ (varifolds with bounded mass) is metrizable
and we choose a metric $\cD$ which induces it.
Moreover, let
$\cV_s\subset \mathcal{V} (M)$ be the 
(closed) subset of stationary varifolds.

\begin{propos}\label{p:stationary}
Let $\Lambda$ be a family of sweepouts which
is closed under homotopies induced by ambient isotopies.
Then there exists a minimizing sequence $\{\{\Gamma_t\}^k\}
\subset\Lambda$ such that, if $\{\Gamma_{t_k}^k\}$ 
is a min-max sequence, then $\cD(\Gamma_{t_k}^k,\cV_s)\to 0$.
\end{propos}

This Proposition is Proposition 4.1 of \cite{CD}.
Though stated for the case $n=2$, this assumption, in fact, is never used in the proof given in that paper. Therefore
we do not include a proof here.

\subsection{Almost mimimizing varifolds}
It is well known that a stationary varifold can be far from regular.
To overcome this issue, we introduce the notion of
almost minimizing varifolds. 

\begin{definition}\label{d:am}
 Let $\eps>0$ and $U\subset M$ open. A boundary 
$\partial \Omega$ in $M$ is called 
\textit{$\eps$-almost minimizing} ($\eps$-a.m.) in $U$ if there is
NO $1$-parameter family of boundaries $\{\partial \Omega_t\}$, 
$t\in [0,1]$ satisfying the following properties:
\begin{eqnarray}
&&\mbox{Properties (s1), (s2), (s3), (sw1) and (sw3)
of Definition \ref{d:sweep} hold;}\label{e:am1}\\
&&\mbox{$\Omega_0=\Omega$ and $\Omega_t\setminus U =
\Omega\setminus U$ for every $t$;}\label{e:am2}\\
&&\mbox{$\cH^n(\partial \Omega_t)\leq \cH^n(\partial \Omega)+
\frac{\eps}{8}$ for all $t\in  [0,1]$;}\label{e:am3}\\
&&\mbox{$\cH^n(\partial \Omega_1)\leq \cH^n(\partial \Omega)-
\eps$.}\label{e:am4}
\end{eqnarray}
A sequence $\{\partial \Omega^k\}$ of hypersurfaces is called 
{\em almost minimizing in $U$} if each 
$\partial \Omega^k$ is $\eps_k$-a.m. 
in $U$ for some sequence $\eps_k\to 0$. 
\end{definition} 

Roughly speaking, $\partial \Omega$ is a.m. if any deformation
which eventually brings down its area is forced to pass
through some surface which has sufficiently larger area. 
A similar notion was introduced
for the first time in the pioneering work of Pitts
and a corresponding one is given in \cite{Sm} using
isotopies (see Section 3.2 of \cite{CD}). 
Following in part Section 5 of \cite{CD} (which uses a 
combinatorial argument inspired by a general one of 
\cite{Alm} reported in \cite{P}), we prove in
Section \ref{s:am} the following existence
result.

\begin{propos}\label{p:almost1}
Let $\Lambda$ be a homotopically closed family of sweepouts. 
There are a function $r:M\to \RR_+$ and a min-max sequence 
$\Gamma^k=\Gamma^k_{t_k}$ such that
\begin{itemize}
\item[(a)] $\{\Gamma^k\}$ is a.m. in every $An\in\An_{r(x)}(x)$
with $x\in M$;
\item[(b)] $\Gamma^k$ converges to a stationary varifold $V$ as 
$k\to \infty$.
\end{itemize}
\end{propos}

In this part we introduce, however, a new ingredient.
The proof of Proposition \ref{p:almost1} has
a variational nature: assuming the nonexistence of such
a minmax sequence we want to show that on an appropriate minimizing
sequence $\{\{\Gamma_t\}^k\}$, the energy $\mathcal{F}
(\{\Gamma_t\}^k)$ can be lowered by a fixed amount, contradicting
its minimality. Note, however, that we have one--parameter families
of surfaces, whereas the variational notion of Definition
\ref{d:am} focuses on a single surface. Pitts (who
in turn has a stronger notion of almost minimality) avoids
this difficulty by considering discretized families and this,
in our opinion, makes his proof quite hard. Instead, our
notion of almost minimality
allows us to stay in the smooth category: the key technical
point is the ``freezing'' presented in Section \ref{ss:freezing}
(cp. with Lemma \ref{l:freezing}).

\subsection{Replacements}\label{ss:rep} We complete the program 
in Sections \ref{s:rep} and \ref{s:reg} 
showing that our notion of 
almost minimality is still sufficient to prove regularity.
As a starting point, as in the theory of Pitts, we consider
{\em replacements}.

\begin{definition}
Let $V\in\cV(M)$ be a stationary varifold and $U\subset M$ be 
an open set. A stationary varifold $V'\in \cV(M)$ is called a 
\textit{replacement for $V$ in $U$} if
$V'=V$ on $M\setminus \bar{U}$, $\|V'\|(M)=\|V\|(M)$
and $V\res U$ is a stable minimal hypersurface $\Gamma$. 
\end{definition}

We show in Section \ref{s:rep} that almost minimizing
varifolds do posses replacements.

\begin{propos}\label{p:replacement}
Let $\{\Gamma^j\}$, $V$ and $r$ be as in Proposition
\ref{p:almost1}. Fix $x\in M$ and consider 
an annulus $An\in \An_{r(x)} (x)$. Then there
are a varifold $\tilde{V}$, a sequence $\{\tilde{\Gamma}^j\}$
and a function $r':M\to \RR_+$ such that
\begin{itemize}
\item[(a)] $\tilde{V}$ is a replacement for $V$ in $An$
and $\tilde{\Gamma}^j$ converges to $\tilde{V}$ in the
sense of varifolds;
\item[(b)] $\tilde{\Gamma}^j$ is a.m. in every $An'\in \An_{r'(y)} (y)$
with $y\in M$;
\item[(c)] $r'(x)=r(x)$.
\end{itemize}
\end{propos}

The strategy of the proof is 
the following. Fix an annulus $An$. We would like to substitute
$\Gamma^j=\partial \Omega^j$ in $An$ with the surface 
minimizing the area among all those
which can be continuously deformed into $\Gamma^j$ according to 
our homotopy class: we could
appropriately call it a solution of 
the {\em $(8j)^{-1}$ homotopic Plateau problem}. As
a matter of fact, we do not know any regularity 
for this problem. However, if we consider
a corresponding minimizing sequence $\partial \{\Omega^{j,k}\}_k$,
we will show that it converges, up to subsequences, to a varifold
$V^j$ which is regular in $An$. This
regularity is triggered by the following observation:
on any sufficiently small ball $B\subset An$, $V^j\res B$ 
is the boundary
of a Caccioppoli set $\Omega^j$ which solves the Plateau problem
in the class $\mathcal{P} (\Omega^j, B)$ 
(in the sense of Theorem \ref{t:DeGiorgi}).

In fact, by standard blowup methods of geometric measure theory, 
$V^j$ is close to a cone in any sufficiently small ball $B= B_r (y)$. 
For $k$ large, the same property holds for $\partial \Omega^{j,k}$. 
Modifying suitably an idea of \cite{Sm}, this property can be used
to show that any (sufficiently regular)
competitor 
$\tilde{\Omega}\in \mathcal{P} (\Omega^{j,k}, B)$ can
be homotopized to $\Omega^{j,k}$ without passing through a surface
of large energy. In other words, minimizing
sequences of the homotopic Plateau problem are in
fact minimizing for the usual Plateau problem
at sufficiently small scales.

Having shown the regularity of $V^j$ in $An$, we use
the Schoen--Simon compactness theorem to show that $V^j$
converges to a varifold $\tilde{V}$ which in $An$ is 
a stable minimal hypersurface. A suitable
diagonal sequence $\Gamma^{j, k(j)}$ gives the
surfaces $\tilde{\Gamma}^j$. 

\subsection{Regularity of $V$}\label{ss:reg} One would
like to conclude that, if $V'$ is a replacement for $V$ 
in an annulus contained in a convex ball,
then $V=V'$ (and hence $V$ is
regular in $An$). However, two stationary varifolds might coincide
outside of a convex set and be different inside:
the standard unique continuation property of classical
minimal surfaces fails in the general case of
stationary varifolds
(see the appendix of \cite{CD} for an example). We need more
information to conclude the regularity of $V$. 
Clearly, applying Proposition
\ref{p:replacement} three times we conclude

\begin{propos}\label{p:three}
Let $V$ and $r$ be as in Proposition \ref{p:almost1}.
Fix $x\in M$ and $An\in \An_{r(x)} (x)$. Then:
\begin{itemize}
\item[(a)] $V$ has a replacement $V'$ in $An$ such that
\item[(b)] $V'$ has a replacement $V''$
in any $An'\in \An_{r (x)} (x)\cup\bigcup_{y\neq x}
\An_{r'(y)} (y)$ such that
\item[(c)] $V''$ has a replacement $V'''$ in any $An''
\in \An_{r''(y)} (y)$ with $y\in M$.
\end{itemize}
$r'$ and $r''$ are positive functions (which
might depend on $V'$ and $V''$).
 \end{propos}

In fact, the process could be iterated infinitely many times.
However, it turns out that three iterations are 
sufficient to prove regularity, as stated in
the following proposition. Its proof is given 
in Section \ref{s:reg}, where
we basically follow \cite{SS} (see also
\cite{CD}). 

\begin{propos}\label{p:reg}
Let $V$ be as in Proposition \ref{p:three}.
Then $V$ is induced by 
a minimal hypersurface $\Sigma$
(in the sense of Definition \ref{d:convention}).
 \end{propos}

\section{The existence of almost mimimizing varifolds}\label{s:am}

In this section we prove Proposition \ref{p:almost1}. At various steps 
in the regularity theory we will have to construct comparison surfaces 
which are deformations of 
a given surface. However, each initial surface will be just a member
of a one--parameter family and in order to exploit our variational
properties we must in fact construct ``comparison families''.  
If we consider a family as a moving surface, if becomes clear that
difficulties come when we try to embed the deformation of a single
``time--slice'' into the dynamics of the family itself. The main
new point of this section is therefore
the following technical lemma, which allows
to use the ``static'' variational principle of Definition \ref{d:am}
to construct a ``dynamic'' competitor.

\begin{lemma}\label{l:freezing}
Let $U\subset\subset U'\subset M$ be two open sets and 
$\{\partial \Xi_t\}_{t\in [0,1]}$
a sweepout. Given an $\eps>0$ and a $t_0\in [0,1]$, assume
$\{\partial \Omega_s\}_{s\in [0,1]}$ is a one--parameter family
of surfaces satisfying \eqref{e:am1}, \eqref{e:am2}, \eqref{e:am3}
and \eqref{e:am4}, with $\Omega = \Xi_{t_0}$. Then there is $\eta>0$, such that the following holds for every $a,b,a',b'$ with
$t_0-\eta \leq b < b'<a'<a\leq t_0+\eta$.
There is a {\em competitor sweepout} $\{\partial \Xi'_t\}_{t\in [0,1]}$ with the following properties:
\begin{itemize}
\item[(a)] $\Xi_t = \Xi'_t$ for $t\in [0,a]\cup [b,1]$ and
$\Xi_t\setminus U'=\Xi'_t\setminus U'$ for $t\in (a,b)$;
\item[(b)] $\cH^n (\partial \Xi'_t)\leq 
\cH^n (\partial \Xi_t) + \frac{\eps}{4}$ for every $t$;
\item[(c)] $\cH^n (\partial \Xi'_t)\leq \cH^n (\partial \Xi_t)
-\frac{\eps}{2}$ for $t\in (a',b')$.
\end{itemize}
Moreover, $\{\partial\Xi'_t\}$ is homotopic to $\{\partial\Xi_t\}$.
\end{lemma}

Bulding on Lemma \ref{l:freezing}, Proposition
\ref{p:almost1} can be proved using a clever combinatorial
argument due to Pitts and Almgren. Indeed, for this
part our proof follows literally the exposition of
Section 5 of \cite{CD}. This section is therefore
split into two parts. In the first one we use the Almgren--Pitts
combinatorial argument to show Proposition \ref{p:almost1}
from Lemma \ref{l:freezing}, which will be proved in
the second.

\subsection{Almost minimizing varifolds}
Before coming to the proof, we introduce some further notation.

\begin{definition}
 Given a pair of open sets $(U^1,U^2)$ we call a hypersurface $\partial
\Omega$ $\eps$-a.m. in $(U^1,U^2)$ if it is $\eps$-a.m. in at least one 
of the two open sets. We denote by $\cC\cO$ the set of pairs 
$(U^1,U^2)$ of open sets with
$$\d(U^1,U^2)\geq 4\min\{\diam(U^1),\diam(U^2)\}.$$
\end{definition}

The following trivial lemma will be of great importance.

\begin{lemma}\label{l:distlem}
 If $(U^1,U^2)$ and $(V^1,V^2)$ are such that
$$
\d(U^1,U^2)\geq 2\min\{\diam(U^1),\diam(U^2)\}\qquad
\d(V^1,V^2)\geq 2\min\{\diam(V^1),\diam(V^2)\}\, ,
$$ 
then there are indices $i,j\in \{1,2\}$ with $\d(U^i,V^j)>0$.
\end{lemma}

We are now ready to state the Almgren--Pitts combinatorial
Lemma: Proposition \ref{p:almost1} is indeed a corollary
of it. 

\begin{propos}[Almgren--Pitts combinatorial Lemma]\label{p:combinatorial}
Let $\Lambda$ be a homotopically closed
 family of sweepouts. There is a min-max 
sequence $\{\Gamma^N\}=\{\partial \Omega^{k(N)}_{t_{k(N)}}\}$ 
such that
\begin{itemize}
\item $\Gamma^N$ converges to a stationary varifold;
\item For any $(U^1, U^2)\in\cC\cO$, $\Gamma^N$ is $1/N$-a.m. 
in $(U^1,U^2)$, for $N$ large enough. 
\end{itemize}
\end{propos}

\begin{proof}[Proof of Proposition \ref{p:almost1}]
 We show that a subsequence of the $\{\Gamma^k\}$ in Proposition 
\ref{p:combinatorial} satisfies the requirements of 
Proposition \ref{p:almost1}. 
For this fix $k\in \NN$ and $r>0$ such that $\Inj(M)>9r>0$. Then, $(B_r(x),M\setminus \overline{B}_{9r} 
(x))\in\cC\cO$ for all $x\in M$. 
Therefore we have that $\Gamma^k$ is (for $k$ large enough)
$1/k$-almost minimizing in $B_r(x)$ or $M\setminus 
\overline{B}_{9r} (x)$. Therefore,
having fixed $r>0$, 
\begin{itemize}
 \item [(a)] either $\{\Gamma^k\}$ is 
(for $k$ large) $1/k$-a.m. in $B_r(y)$ for every $y\in M$;
 \item [(b)] or there are a (not relabeled) subsequence
$\{\Gamma^k\}$ and a sequence $\{x^k_r\}\subset M$ such that 
$\Gamma^k$ is $1/k$-a.m. in $M\setminus \overline{B}_{9r} (x^k_r)$.
\end{itemize}
If for some $r>0$ (a) holds, we clearly have a sequence
as in Proposition \ref{p:almost1}.
Otherwise there are a subsequence of $\{\Gamma^k\}$,
not relabeled, and a collection of points $\{x^k_j\}_{k,j\in \N}\subset M$
such that 
\begin{itemize}
 \item for any fixed $j$, $\Gamma^k$ is $1/k$-a.m. in $M\setminus \overline{B}_{1/j} (x^k_j)$ for $k$ large enough;
 \item $x^k_j\to x_j$ for $k\to \infty$ and $x_j\to x$ for $j\to \infty$.
\end{itemize}
We conclude that, for any $J$, there is $K_J$ such that
$\Gamma^k$ is $1/k$--a.m. in $M\setminus \overline{B}_{1/J} (x)$
for all $k\geq K_J$.
Therefore, if $y\in M\setminus \{x\}$, we choose $r (y)$
such that $B_{r(y)}\subset\subset M\setminus \{x\}$, whereas
$r(x)$ is chosen arbitrarily. It follows that $An\subset\subset M\setminus\{x\}$, for
any $An\in \An_{r(z)} (z)$ with $z\in M$.
Hence, $\{\Gamma^k\}$ is $1/k$-a.m. 
in $An$, provided $k$ is large enough, which completes the proof
of the Proposition.
\end{proof}

\begin{proof}[Proof of Proposition \ref{p:combinatorial}]
We start by picking a minimizing sequence $\{\{\Gamma_t\}^k\}$ 
satisfying the requirements of Proposition \ref{p:stationary} and such 
that $\cF(\{\Gamma_t\}^k)<m_0+\frac{1}{8k}$. We then
assert the following claim, which clearly implies the Proposition.
\begin{Cla}
For $N$ large enough, there exists $t_N\in [0,1]$ such that 
$\Gamma^N:=\Gamma^N_{t_N}$ is $\frac{1}{N}$-a.m. in all 
$(U^1,U^2)\in \cC\cO$ and $\cH^n(\Gamma^N)\geq m_0-\frac{1}{N}$.
\end{Cla}
Define
$$
K_N\;:=\;\left\{t\in [0,1]:\cH^n(\Gamma^N_t)\geq m_0-\frac{1}{N}\right\}.
$$
Assume the claim is false. Then there is a sequence $\{N_k\}$
such that the assertion of the claim is violated
for every $t\in K_{N_k}$. By a slight abuse of notation, we
do not relabel the corresponding subsequence and from
now on we drop the super- and subscripts $N$.

Thus, for
every $t\in K$ we get a pair $(U_{1,t}, U_{2,t})\in \cC\cO$
and two families $\{\partial \Omega_{i,t, \tau}\}^{i\in\{1,2\}}_{\tau\in [0,1]}$
such that
\begin{itemize}
 \item[(i)] $\partial \Omega_{i,t,\tau}\cap (U_{i,t})^c= 
\partial \Omega_t\cap (U_{i,t})^c$
 \item[(ii)] $\partial \Omega_{i,t,0}=\partial \Omega_t$
 \item[(iii)] $\cH^n(\partial \Omega_{i,t,\tau})\leq 
\cH^n (\partial \Omega_t)+\frac{1}{8N}$
 \item[(iv)] $\cH^n(\partial \Omega_{i,t,1})\leq \cH^n (\partial
\Omega_t)-\frac{1}{N}$. 
\end{itemize}
For every $t\in K$ and every $i\in \{1,2\}$, we choose
$U'_{i,t}$ such that $U_{i,t}\subset\subset U'_{i,t}$ and
$$
\d (U'_{1,t}, U'_{2,t})\geq 2 \min \{\diam (U'_{1,t}),
\diam (U'_{2,t})\}
$$
Then we apply
Lemma \ref{l:freezing} with $\Xi_t = \Omega_t$, $U=U_{i,t}$,
$U'=U'_{i,t}$ and
$\Omega_\tau = \Omega_{i,t, \tau}$. 
Let $\eta_{i,t}$ be the corresponding constant $\eta$ given
by Lemma \ref{l:freezing} and let $\eta_t = \min\{\eta_{1,t},
\eta_{2,t}\}$. 

Next, cover $K$ with intervals $I_i = (t_i - \eta_i,
t_i + \eta_i)$ in such a way that: 
\begin{itemize}
\item $t_i + \eta_i< t_{i+2}-\eta_{i+2}$ for every $i$;
\item $t_i\in K$ and $\eta_i< \eta_{t_i}$.
\end{itemize}

\medskip

{\bf Step 1: Refinement of the covering.}
We are now going to refine the covering $I_i$ to
a covering $J_l$ such that: 
\begin{itemize}
\item $J_l\subset I_i$ for some $i(l)$;
\item there is a choice of a $U_l$ such that
$U'_l\in \{U'_{1,t_{i(l)}}, U'_{2, t_{i(l)}}\}$
and 
\begin{equation}\label{e:disjoint}
\d (U'_i, U'_j)>0 \qquad \mbox{if $\overline{J}_i\cap 
\overline{J}_j\neq\emptyset$;}
\end{equation}
\item each point $t\in [0,1]$ is contained in at most two
of the intervals $J_l$.
\end{itemize}
The choice of our refinement is in fact quite obvious.
We start by choosing $J_1=I_1$. Using Lemma \ref{l:distlem}
we choose indices $r,s$ such that $\dist (U'_{r, t_1}, U'_{s, t_2})>0$.
For simplicity we can assume $r=s=1$. We then
set $U'_1=U'_{1, t_1}$. Next, we consider two indices
$\rho, \sigma$ such that $\d (U'_{\rho, t_2}, U'_{\sigma, t_3})>0$.
If $\rho=1$, we then set $J_2=I_2$ and $U'_2 = U'_{1,t_2}$.
Otherwise, we cover $I_2$ with two open intervals $J_2$ and $J_3$,
with the property that $\overline{J}_2$ is disjoint from
$\overline{I}_3$ and $\overline{J}_3$ is disjoint from $\overline{I}_1$.
We then choose $U'_2= U'_{1, t_2}$ and $U'_3=U'_{2, t_2}$.
From this we are ready to proceed inductively. Note therefore
that, in our refinement of the covering, 
each interval $I_j$ with $j\geq 2$ get either ``split into
two halves'' or remains the same (cp. with Figure 1, left).   

Next, fixing the notation $(a_i, b_i)=J_i$, we choose $\delta>0$
with the property: 
\begin{itemize}
\item[(C)] Each $t\in K$ is contained in at least 
one segment $(a_i+\delta, b_i-\delta)$ (cp. with Figure 1, right).
\end{itemize}

\begin{figure}[htbp]
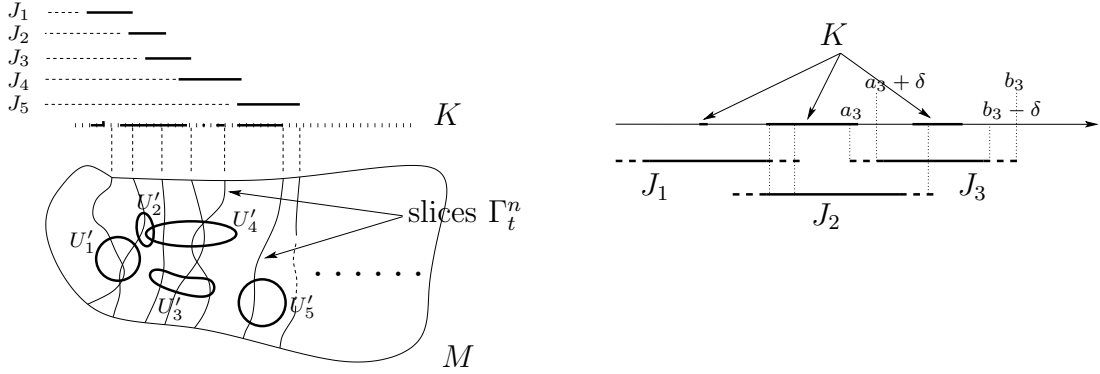

\begin{center}
\input{Pic1.pstex_t} 
\input{Pic1.1.pstex_t}
    \caption{The left picture shows the refinement of the
covering. We split $I_2$ into $J_2\cup J_3$ because
$U'_4=U'_{1,t_3}$ intersects $U'_2= U'_{1,t_2}$. The refined
covering has the property that $U'_i\cap U'_{i+1}=\emptyset$. 
In the right picture the segments $(a_k, b_k)=J_k$ and
$(a_k+\delta, b_k-\delta)$. Any point $\tau\in K$ belongs
to at least one $(a_i+\delta, b_i-\delta)$ and to
at most one $J_j\setminus (a_j+\delta, b_j-\delta)$.}
    \label{f:refinement}
\end{center}
\end{figure}

\medskip

{\bf Step 2: Conclusion.} We now apply Lemma \ref{l:freezing}
to conclude the existence of a family $\{\partial \Omega_{i,t}\}$
with the following properties:
\begin{itemize}
\item $\Omega_{i,t}=\Omega_t$ if $t\not\in (a_i, b_i)$
and $\Omega_{i,t}\setminus U'_i = \Omega_t\setminus U'_i$
if $t\in (a_i, b_i)$;
\item $\cH^n (\partial \Omega_{i,t}) \leq \cH^n (\partial\Omega_t)
+ \frac{1}{4N}$ for every $t$;
\item $\cH^n (\partial \Omega_{i,t})\leq \cH^n (\partial \Omega_t)
- \frac{1}{2N}$ if $t\in (a_i+\delta, b_i-\delta)$.
\end{itemize}
Note that, if $t\in (a_i, b_i)\cap (a_j, b_j)$,
then $j=i+1$ and in fact $t\not\in (a_k, b_k)$ for
$k\neq i, i+1$.
Moreover, $\dist (U'_i, U'_{i+1})>0$. Thus, we can define 
a new sweepout $\{\partial \Omega'_t\}_{t\in [0,1]}$
\begin{itemize}
\item $\Omega'_t=\Omega_t$ if $t\not\in \cup J_i$;
\item $\Omega'_t =\Omega_{i,t}$ if $t$ is contained
in a single $J_i$; 
\item $\Omega'_t = \left[\Omega_t\setminus (U'_i\cup U'_{i+1})\right]
\cup \left[\Omega_{i,t}\cap U'_i\right] \cup \left[\Omega_{i+1,t}\cap U'_{i+1}\right]$
if $t\in J_i\cap J_{i+1}$.
\end{itemize}
In fact, it is as well easy to check that 
$\{\partial \Omega'_t\}_{t\in [0,1]}$ is homotopic
to $\{\partial \Omega_t\}$ and hence belongs to $\Lambda$.

Next, we want to compute $\mathcal{F} (\{\partial \Omega'_t\})$.
If $t\not \in K$, then $t$ is contained in at most 
two $J_i$'s, and hence $\partial \Omega'_t$
can loose at most $2\cdot \textstyle{\frac{1}{4N}}$
in area:
\begin{equation}\label{e:outside}
t\not\in K \quad \Rightarrow\quad  \cH^n (\partial \Omega'_t)
\;\leq\; \cH^n (\partial \Omega_t) + \frac{1}{2N}
\;\leq\; m_0 (\Lambda) - \frac{1}{2N}\, .
\end{equation}
If $t\in K$, then $t$ is contained in at least
one segment $(a_i+\delta, b_i-\delta)\subset J_i$ and in at most
a second segment $J_l$. Thus, the area of $\partial \Omega'_t$
gains at least $\textstyle{\frac{1}{2N}}$
in $U'_i$ and looses at most $\textstyle{\frac{1}{4N}}$
in $U'_l$. Therefore we conclude
\begin{equation}\label{e:inside}
t\in K \quad \Rightarrow\quad  \cH^n (\partial \Omega'_t)
\;\leq\; \cH^n (\partial \Omega_t) - \frac{1}{4N}
\;\leq\; m_0 (\Lambda) - \frac{1}{8N}\, .
\end{equation}
Hence $\cF (\{\partial\Omega'_t\}) \leq m_0 (\Lambda)
- (8N)^{-1}$, which is a contradiction
to $m_0 (\Lambda) = \inf_\Lambda \cF$.
\end{proof}

\subsection{Proof of Lemma \ref{l:freezing}}\label{ss:freezing}
{\bf Step 1: Freezing.} First of all we choose open sets $A$ and $B$ such that
\begin{itemize}
\item $U\subset\subset A \subset\subset B \subset\subset U'$;
\item $\partial \Xi_{t_0}\cap C$ is a smooth surface,
where $C=B\setminus \overline{A}$.
\end{itemize}
This choice is clearly possible since there are only finitely many singularities of
$\partial \Xi_{t_0}$. Next, we
fix two smooth functions $\varphi_A$ and $\varphi_B$ such that
\begin{itemize}
\item $\varphi_A+\varphi_B=1$;
\item $\varphi_A\in C^\infty_c (B)$, $\varphi_B\in C^\infty_c (M\setminus
\overline{A})$.
\end{itemize}
Now, we fix normal coordinates $(z,\sigma)\in \partial \Xi_{t_0}\cap C\times (-\delta, \delta)$
in a regular $\delta$--neighborhood of $C\cap \partial \Xi_{t_0}$. 
Because of the convergence of $\Xi_t$ to $\Xi_{t_0}$,
we can fix $\eta>0$ and an open $C'\subset C$,
such that the following holds for every $t\in (t_0-\eta, t_0+\eta)$:
\begin{itemize}
\item $\partial \Xi_t\cap C$ is the graph of a function $g_t$
over $\partial \Xi_{t_0}\cap C$;
\item $\Xi_t\cap C\setminus C' = \Xi_{t_0}\cap C\setminus C'$;
\item $\Xi_t\cap C' = \{(z,\sigma): \,\sigma<g_t(z)\}\cap C'$,
\end{itemize}
(cp. with Figure 2). Obviously,
$g_{t_0}\equiv 0$. We next introduce the functions
\begin{equation}\label{e:interpolants}
g_{t,s, \tau}\;:=\; \varphi_B g_t + \varphi_A ((1-s) g_t + s g_{\tau})\, 
\qquad t, \tau\in (t_0-\eta, t_0+\eta), s\in [0,1].
\end{equation}
Since $g_t$ converges smoothly to $g_{t_0}$ as $t\to t_0$, by choosing
$\eta$ arbitrarily small, we can make $\sup_{s, \tau} \|g_{t,s, \tau}-g_t\|_{C^1}$
arbitrarily small. Next, if we express the area of the graph of a function
$g$ over $\partial \Xi_{t_0}\cap C$ as an integral functional of $g$,
this functional depends
obviously only on $g$ and its first derivatives.
Thus, if $\Gamma_{t,s, \tau}$ is the graph of $g_{t,s, \tau}$, then we can
choose $\eta$ so small that
\begin{equation}\label{e:estimate1}
\max_s \cH^n (\Gamma_{t,s, \tau}) \;\leq\; \cH (\partial \Xi_t \cap C) + \frac{\eps}{16}\, .
\end{equation}
Now, given $t_0-\eta<a<a'<b'<b<t_0+\eta$, we choose $a''\in (a, a')$ and
$b''\in (b',b)$ and fix:
\begin{itemize}
\item a smooth function $\psi: [a,b]\to [0,1]$
which is identically equal to $0$ in a neighborhood of $a$ and $b$
and equal to $1$ on $[a'', b'']$;
\item a smooth function $\gamma: [a,b]\to [t_0-\eta, t_0+\eta]$ which
is equal to the identity in a neighborhood of $a$ and $b$ and indentically
$t_0$ in $[a'', b'']$.
\end{itemize}
Next, define the family of open sets $\{\Delta_t\}$ as follows:
\begin{itemize}
\item $\Delta_t = \Xi_t$ for $t\not\in [a,b]$;
\item $\Delta_t\setminus \overline{B} = \Xi_t\setminus \overline{B}$
for all $t$;
\item $\Delta_t\cap A = \Xi_{\gamma (t)}\cap A$ for $t\in [a,b]$;
\item $\Delta_t\cap C\setminus C'= \Xi_{t_0}\cap C \setminus C'$ for $t\in [a,b]$;
\item $\Delta_t\cap C' = \{(z,\sigma): \sigma < g_{t, \psi (t), \gamma(t)} (z)\}$ for $t\in [a,b]$.
\end{itemize}
Note that $\{\partial \Delta_t\}$ is in fact a sweepout homotopic to
$\partial \Xi_t$. In addition: 
\begin{itemize}
\item $\Delta_t=\Xi_t$ if $t\not\in [a,b]$,
and $\Delta_t$ and $\Xi_t$ coincide outside of $B$ (and hence outside of $U'$) for {\em every} $t$;
\item $\Delta_t\cap A=\Xi_{\gamma (t)}\cap A$ for $t\in [a,b]$
(and hence $\Delta_t\cap U=\Xi_{\gamma (t)}\cap U$).
\end{itemize}
Therefore, $\Delta_t\cap U=\Xi_{t_0}\cap U$ for $t\in [a'',b'']$, i.e.
$\Delta_t\cap U$ is {\em frozen} in the interval $[a'',b'']$.
Moreover, because of \eqref{e:estimate1}, 
\begin{equation}\label{e:layer}
\cH^n (\partial \Delta_t\cap C) \leq \cH^n (\partial \Xi_t\cap C)+\frac{\eps}{16}
\qquad \mbox{for $t\in [a,b]$.}
\end{equation}

\medskip

{\bf Step 2: Dynamic competitor.}
Next, fix a smooth function $\chi: [a'', b'']\to [0,1]$ which is identically $0$ in
a neighborhood of $a''$ and $b''$ and which is identically $1$ on $[a',b']$. We
set
\begin{itemize}
\item $\Xi'_t = \Delta_t$ for $t\not\in [a'', b'']$;
\item $\Xi'_t\setminus A = \Delta_t\setminus A$ for $t\in [a'', b'']$;
\item $\Xi'_t\cap A = \Omega_{\chi (t)}\cap A$ for $t\in [a'',b'']$.
\end{itemize}
The new family $\{\partial \Xi'_t\}$ is also a sweepout, obviously homotopic
to $\{\partial \Delta_t\}$ and hence homotopic to $\{\partial \Xi_t\}$. We
next estimate $\cH^n (\partial \Xi'_t)$.
For $t\not\in [a,b]$, $\Xi'_t \equiv \Xi_t$ and hence
\begin{equation}\label{e:out}
\cH^n (\partial \Xi'_t) \;=\; \cH^n (\partial \Xi_t)
\qquad\mbox{for $t\not\in [a,b]$.}
\end{equation}
For $t\in [a,b]$, we anyhow have $\Xi'_t=\Xi_t$ on $M\setminus B$ and
$\Xi'_t= \Delta_t$ on $C$. This shows the property $(a)$ of the lemma. Moreover, for $t\in [a,b]$ we have
\begin{eqnarray}
\cH^n (\partial \Xi'_t)
- \cH^n (\partial \Xi_t) &\leq& [\cH^n (\partial \Delta_t\cap C)
- \cH^n (\partial \Xi_t\cap C)] + [\cH^n (\partial \Xi'_t \cap A)
 -\cH^n (\partial \Xi_t\cap A)]\, \nonumber\\
&\stackrel{\eqref{e:layer}}{\leq}& \frac{\eps}{16} + 
[\cH^n (\partial \Xi'_t \cap A) -\cH^n (\partial \Xi_t\cap A)]\label{e:break}.
\end{eqnarray}
To conclude, we have to estimate the part in $A$ in the time interval $[a,b]$. We have to consider several cases separately. 
\begin{itemize}
 \item[(i)] Let $t\in [a,a'']\cup [b'',b]$. Then $\Xi'_t\cap A = \Delta_t
\cap A = \Xi_{\gamma (t)}\cap A$. However, $\gamma (t), t\in (t_0-\eta,
t_0+\eta)$ and, having chosen $\eta$ sufficiently small, we can assume
\begin{equation}\label{e:oscil}
|\cH^n (\partial \Xi_s\cap A) - \cH^n (\partial \Xi_\sigma \cap A)|
\;\leq\; \frac{\eps}{16} \qquad \mbox{for every $\sigma,s\in (t_0-\eta, t_0+\eta)$}
\end{equation}
(note: this choice of $\eta$ is independent of $a$ and $b$!).
Thus, using \eqref{e:break}, we get
\begin{equation}\label{e:aa''}
\cH^n (\partial \Xi'_t)
\;\leq\; \cH^n (\partial \Xi_t) + \frac{\eps}{8}.
\end{equation}
\item[(ii)] Let $t\in [a'',a']\cup [b'',b']$. Then $\partial \Xi'_t\cap A
= \partial \Omega_{\chi (t)}\cap A$. Therefore we can write, using \eqref{e:break},
\begin{eqnarray}
\cH^n (\partial \Xi'_t)-\cH^n (\partial \Xi_t)
&\leq& \frac{\eps}{16}
+ [\cH^n (\partial \Xi_{t_0}\cap A) -\cH^n (\partial \Xi_t\cap A)]\nonumber\\
&&\quad\;\,+\, [\cH^n (\partial \Omega_{\chi (t)}\cap A) - \cH^n (\partial \Xi_{t_0}\cap A)]\nonumber\\
&\stackrel{\eqref{e:oscil},\eqref{e:am3}}{\leq}&
\frac{\eps}{16}+\frac{\eps}{16}+\frac{\eps}{8}
\;=\; \frac{\eps}{4}.
\label{e:a''a'}
\end{eqnarray}
\item[(iii)] Let $t\in [a',b']$. Then we have $\Xi'_t\cap A = \Omega_1\cap A$. Thus, again using \eqref{e:break},
\begin{eqnarray}
\cH^n (\partial \Xi'_t)-\cH^n (\partial \Xi_t)
&\leq& \frac{\eps}{16}
+ [\cH^n (\partial \Omega_1\cap A) -\cH^n (\partial \Xi_{t_0}\cap A)]\\
&&\quad\;\,+ \,[\cH^n (\partial \Xi_{t_0}\cap A) - \cH^n (\partial \Xi_t\cap A)]\nonumber\\
&\stackrel{\eqref{e:am4},\eqref{e:oscil}}{\leq}&
\frac{\eps}{16}-\eps+\frac{\eps}{16}
\;<\; -\frac{\eps}{2}.
\label{e:a'b'}
\end{eqnarray}
\end{itemize}
Gathering the estimates \eqref{e:out}, \eqref{e:aa''}, \eqref{e:a''a'} and \eqref{e:a'b'}, 
we finally obtain the properties $(b)$ and $(c)$ of the lemma. This finishes the proof.

\begin{figure}[htbp]
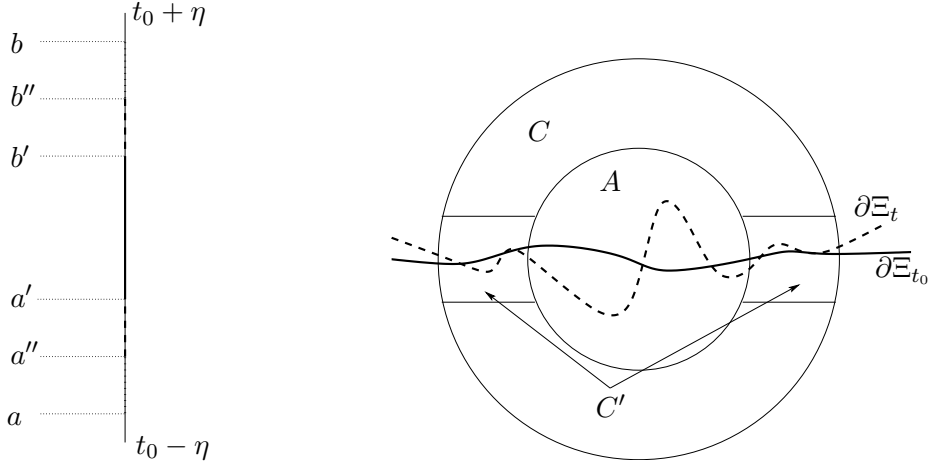

\begin{center}
\input{Pic2.pstex_t} 
\input{Pic2.1.pstex_t}
    \caption{The left picture shows the intervals
involved in the construction. If we focus on the
smaller set $A$, then: the sets $\Xi'_t$ coincide
with $\Delta_t$ and evolve from
$\Xi_a$ to $\Xi_{t_0}$ (resp. $\Xi_{t_0}$ to $\Xi_b$)
in $[a,a'']$ (resp. $[b'', b]$); they then evolve
from $\Xi_{t_0}$ to $\Omega_1$ (resp. $\Omega_1$
to $\Xi_{t_0}$) in $[a'',a']$ (resp. $[b', b'']$). 
On the right picture, the sets in the region $C$. 
Indeed, the evolution takes place in the region
$C'$ where we patch smoothly $\Xi_{t_0}$
with $\Xi_{\gamma (t)}$ into the sets $\Delta_t$.}
    \label{f:freezing}
\end{center}
\end{figure}

\section{The existence of replacements}
\label{s:rep}

In this section we fix $An\in \An_{r(x)} (x)$
and we prove the conclusion of Proposition \ref{p:replacement}.

\subsection{Setting}\label{ss:rep_setting} For every $j$, consider
the class $\mathcal{H} (\Omega^j, An)$
of sets $\Xi$ such that 
there is a family $\{\Omega_t\}$ satisfying
$\Omega_0=\Omega^j$,
$\Omega_1=\Xi$,
\eqref{e:am1}, \eqref{e:am2} and \eqref{e:am3} for $\eps=\frac{1}{j}$ and $U=An$.
Consider next a sequence $\Gamma^{j,k} = \partial \Omega^{j,k}$
which is minimizing for the perimeter in the class $\mathcal{H}
(\Omega^j, An)$: this is the minimizing sequence for 
the {\em $(8j)^{-1}$--homotopic Plateau problem}
mentioned in Subsection \ref{ss:rep}. 
Up to subsequences, we can assume that 
\begin{itemize}
\item $\Omega^{j,k}$ converges to a Caccioppoli set
$\tilde{\Omega}^j$;
\item $\Gamma^{j,k}$ converges to a varifold $V^j$;
\item $V^j$ (and a suitable diagonal sequence 
$\tilde{\Gamma}^j = \Gamma^{j, k(j)}$)
converges to a varifold $\tilde{V}$.
\end{itemize}
The proof of Proposition \ref{p:replacement} will
then be broken into three steps. In the first one we show

\begin{lemma}\label{l:concon} For every $j$ and
every $y\in An$ there is a ball $B = B_\rho (y)\subset An$ 
and a $k_0\in \NN$ with the following property. 
Every open set $\Xi$ such that
\begin{itemize}
\item $\partial \Xi$ is smooth except for a finite set,
\item $\Xi\setminus B = \Omega^{j,k}\setminus B$,
\item and $\cH^n (\partial \Xi) < \cH^n (\partial 
\Omega^{j,k})$,     
\end{itemize}
belongs to $\mathcal{H} (\Omega^j, An)$ if $k\geq k_0$.
\end{lemma}

In the second step we use Lemma \ref{l:concon} and
Theorem \ref{t:DeGiorgi} to show:

\begin{lemma}\label{l:min}
$\partial \tilde{\Omega}^j\cap An$ is a stable minimal hypersurface
in $An$ and $V^j \res An = \partial \tilde{\Omega}^j\res An$.
\end{lemma}

Recall that in this section we use the convention
of Definition \ref{d:convention}. 
In the third step we use Lemma \ref{l:min} to conclude that
the sequence $\tilde{\Gamma}^j$ and the varifold 
$\tilde{V}$ meet the requirements
of Proposition \ref{p:replacement}.

\subsection{Proof of Lemma \ref{l:concon}}
\label{ss:concon}
The proof of the lemma is achieved 
by exhibiting a suitable
homotopy between $\Omega^{j,k}$ and $\Xi$. The
key idea is:
\begin{itemize}
\item First deform $\Omega^{j,k}$ to
the set $\tilde{\Omega}$ which is
the union of $\Omega^{j,k}\setminus B$ and
the cone with vertex $y$ and base $\Omega^{j,k}\cap \partial B$;
\item Then deform $\tilde{\Omega}$ to $\Xi$.
\end{itemize}
The surfaces of the homotopizing family 
do not gain too much in area, provided
$B=B_\rho (y)$ is sufficiently small and $k$ sufficiently
large: in this case the area of the surface $\Gamma^{j,k}\cap
B$ will, in fact, be close to the area of the cone.
This ``blow down--blow up'' procedure is an idea which
we borrow from \cite{Sm} (see Section 7 of \cite{CD}).

\begin{proof}[Proof of Lemma \ref{l:concon}] We fix
$y\in An$ and $j\in \NN$. 
Let $B= B_\rho (y)$ with $B_{2\rho} (y) \subset An$ 
and consider
an open set $\Xi$ as in the statement of the Lemma.
The choice of the radius of the ball $B_\rho (y)$
and of the constant $k_0$
(which are both independent of the set $\Xi$)
will be determined at the very end of the proof.

\medskip

{\bf Step 1: Stretching $\Gamma^{j,k}\cap \partial B_r (y)$.} 
First of all, we choose $r\in (\rho, 2\rho)$ 
such that, for every $k$,
\begin{equation}\label{e:req1}
\mbox{$\Gamma^{j,k}$ is
regular in a neighborhood of $\partial B_r (y)$ and intersects
it transversally}.
\end{equation}
In fact, since each $\Gamma^{j,k}$ has 
finitely many singularities, Sard's Lemma implies
that \eqref{e:req1} is satisfied by a.e. $r$.
We assume moreover that $2\rho$ is smaller than the
injectivity radius. For each $z\in \overline{B}_r (y)$ 
we consider the closed geodesic arc $[y,z]\subset 
\overline{B}_r (y)$ joining $y$ and $z$. As usual,
$(y,z)$ denotes $[y,z]\setminus \{y,z\}$.
We let $K$ be the open cone consisting 
\begin{equation}\label{e:cone}
K \;=\; \bigcup_{z\in \partial B \cap \Omega^{j,k}} (y,z)\, .
\end{equation}
We now show that $\Omega^{j,k}$ can be homotopized
through a family $\tilde{\Omega}_t$ 
to a $\tilde{\Omega}_1$ in such a way that
\begin{itemize}
\item $\max_t \cH^n (\partial \tilde{\Omega}_t) - 
\cH^n (\partial \Omega^{j,k})$ 
can be made arbitrarily small;
\item $\tilde{\Omega}_1$ coincides with $K$ in a neighborhood
of $\partial B_r (y)$.
\end{itemize}

First of all consider a smooth function $\varphi:[0, 2\rho]
\to [0,2\rho]$, with
\begin{itemize}
\item $|\varphi(s) -s|\leq\eps$ and $0\leq \varphi'\leq 2$;
\item $\varphi (s)=s$ if $|s-r|>\eps$
and $\varphi\equiv r$ in a neighborhood of $r$.
\end{itemize}
Set $\Phi (t,s) := (1-t) s + t \varphi (s)$.
Moreover, for every $\lambda\in [0,1]$ and every
$z\in \overline{B}_r (y)$ let $\tau_\lambda (z)$
be the point $w\in [y,z]$ with $\dist\, (y,w)=\lambda\, 
\dist\, (y,z)$. For $1<\lambda<2$, we can still define $\tau_\lambda(z)$ to be the corresponding point on the geodesic that is the extension of $[y,z]$. (Note that by the choice of $\rho$ this is well defined.) We are now ready to define
$\tilde{\Omega}_t$ (cp. with Figure 3, left).
\begin{itemize}
\item $\tilde{\Omega}_t\setminus An (y, r-\eps, r+\eps)
= \Omega^{j,k} \setminus An (y, r-\eps, r+\eps)$;
\item $\tilde{\Omega}_t \cap \partial B_s (y)
= \tau_{s/\Phi (t,s)} (\Omega^{j,k}\cap \partial B_{\Phi (t,s)})$
for every $s\in (r-\eps, r+\eps) $.
\end{itemize}

Thanks to \eqref{e:req1}, for $\eps$ sufficiently small
$\tilde{\Omega}_t$ has the desired properties. 
Moreover, since $\Xi$ coincides with $\Omega^{j,k}$
on $M\setminus B_\rho (y)$, the same argument 
can be applied to $\Xi$. This shows that
\begin{equation}\label{e:wlog}
\mbox{w.l.o.g. we can assume $K=\Xi=\Omega^{k,j}$
in a neighborhood of $\partial B_r (y)$,}
\end{equation}
(cp. with Figure 3, right).
\medskip

{\bf Step 2: The homotopy} 
We then consider the following family of open
sets $\{\Omega_t\}_{t\in [0,1]}$:
\begin{itemize}
\item $\Omega_t\setminus \overline{B}_r (y) = \Omega^{j,k}
\setminus \overline{B}_r (y)$ for every $t$;
\item $\Omega_t \cap An (y, |1-2t| r, r)
= K\cap An (y, |1-2t| r, r)$ for every $t$;
\item $\Omega_t \cap \overline{B}_{(1-2t)r} (y)
= \tau_{1-2t} (\Omega^{k,j}\cap \overline{B}_r (y))$
for $t\in [0,\frac{1}{2}]$;
\item $\Omega_t\cap \overline{B}_{(2t-1)r} (y)
= \tau_{2t-1} (\Xi\cap \overline{B}_r (y))$
for $t\in [\frac{1}{2}, 1]$.
\end{itemize}

\begin{figure}[htbp]
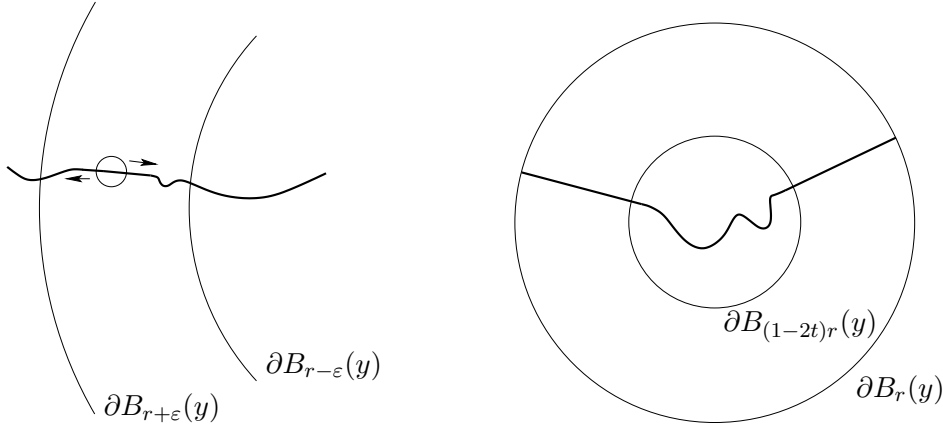

\begin{center}
\input{Pic3.pstex_t} 
\input{Pic3.1.pstex_t}
    \caption{The left picture illustrates
the stretching of $\Gamma^{j,k}$ into
a cone--like surface in a neighborhood
of $\partial B_r (y)$. The right picture shows
a slice $\Omega_t\cap B_r (y)$ for $t\in (0,1/2)$.} 
    \label{f:concon}
\end{center}
\end{figure}

Because of \eqref{e:wlog}, this family satisfies (s1)--(s3),
(sw1) and (sw3). It remains to check, 
\begin{equation}\label{e:am3bis}
\max_t \cH^n (\partial \Omega_t)
\;\leq\; \cH^n (\partial \Omega^{j,k}) +\frac{1}{8j}\, 
\qquad \forall k\geq k_0
\end{equation}
for a suitable choice of $\rho$, $r$ and $k_0$.

First of all we observe that, by the smoothness of $M$,
there are constants $\mu$ and $\rho_0$, depending
only on the metric, such that the following
holds for every $r<2\rho<2 \rho_0$ and $\lambda\in [0,1]$:
\begin{eqnarray}
&&\cH^n (K) \;\leq\; \mu r \cH^{n-1} (\partial
\Omega^{j,k}\cap \partial B_r (y))\label{e:cone_est}\\
&&\cH^n ([\partial (\tau_\lambda (\Omega^{j,k}
\cap \overline{B}_r (y)))]\cap B_{\lambda r} (y))\;\leq\; 
\mu\cH^n (\partial \Omega^{j,k} \cap B_r (y))\label{e:shrink1}\\
&&\cH^n ([\partial (\tau_\lambda (\Xi
\cap \overline{B}_r (y)))]\cap B_{\lambda r} (y))\;
\leq\; \mu\cH^n (\partial \Xi \cap B_r (y))\label{e:shrink2}\\
&&\int_0^{2\rho} \cH^{n-1} (\partial \Omega^{j,k}
\cap \partial B_\tau (y))\, d\tau
\;\leq\; \mu \cH^n (\partial \Omega^{j,k}\cap B_{2\rho} (y))
\label{e:coarea}\, .
\end{eqnarray}
In fact, for $\rho$ small, $\mu$ will be close to $1$. \eqref{e:cone_est}, \eqref{e:shrink1} and \eqref{e:shrink2}
give the obvious estimate
\begin{equation}\label{e:est_max}
\max_t \cH^n (\partial \Omega_t)-\cH^n (\partial \Omega^{j,k})
\;\leq\; \mu\cH^n (\partial \Omega^{j,k}\cap B_{2\rho} (y))
+ \mu r \cH^{n-1} (\partial \Omega^{j,k}\cap \partial B_r (y))\, .
\end{equation}
Moreover, by \eqref{e:coarea} we can
find $r\in (\rho, 2\rho)$ which,
in addition to \eqref{e:est_max}, satisfies
\begin{equation}\label{e:req2}
\cH^{n-1} (\partial \Omega^{j,k}\cap \partial
B_r (y))\;\leq\; \frac{2\mu}{\rho} \cH^n (\partial \Omega^{j,k}
\cap B_{2\rho} (y))\, .
\end{equation}
Hence, we conclude
\begin{equation}\label{e:est_max2}
\max_t \cH^n (\partial \Omega_t)
\;\leq\; \cH^n (\partial \Omega^{j,k})
+ (\mu+2\mu^2) \cH^n (\partial \Omega^{j,k}\cap B_{2\rho} (y))\, .
\end{equation}
Next, by the convergence of $\Gamma^{j,k}=\partial\Omega^{j,k}$
to the stationary varifold $V^j$, we can choose $k_0$ such
that 
\begin{equation}\label{e:req3}
\cH^n (\partial \Omega^{j,k}\cap B_{2\rho} (y))
\;\leq\; 2 \|V^j\| (B_{4\rho} (y)) \qquad \mbox{for $k\geq k_0$.}
\end{equation}
Finally, by the monotonicity formula,
\begin{equation}\label{e:est_max3}
\|V^j\| (B_{4\rho} (y))\leq C_M \|V^j\| (M) \rho^n\, .
\end{equation}
We are hence ready to specify the choice of the
various parameters.
\begin{itemize}
\item We first determine the constants $\mu$ and $\rho_0<
\Inj (M)$ (which depend only on $M$) which
guarantee \eqref{e:cone_est}, \eqref{e:shrink1},
\eqref{e:shrink2} and \eqref{e:coarea};
\item We subsequently choose $\rho<\rho_0$ so small that 
$2 (\mu+ 2 \mu^2) C_M \|V^j\| (M) \rho^n < (8j)^{-1}$;
and $k_0$ so that \eqref{e:req3} holds.
\end{itemize}
At this point $\rho$ and $k$ are
fixed and, choosing $r\in (\rho, 2\rho)$ satisfying
\eqref{e:req1} and \eqref{e:req2}, we construct
$\{\partial \Omega_t\}$ as above, concluding
the proof of the lemma.

\subsection{Proof of Lemma \ref{l:min}}
Fix $j\in \NN$ and $y\in An$ and
let $B= B_\rho (y)\subset An$ be the ball given by Lemma
\ref{l:concon}. We claim that $\tilde{\Omega}^j$
minimizes the perimeter in the class $\mathcal{P}
(\tilde{\Omega}^j, B_{\rho/2} (y))$.
Assume, by contradiction, that $\Xi$ is
a Caccioppoli set with $\Xi\setminus B_{\rho/2} (y)
= \tilde{\Omega}\setminus B_{\rho/2} (y)$ and
\begin{equation}\label{e:competitor}
\per (\Xi) \;<\; \per (\tilde{\Omega}^j)
-\eta\, .
\end{equation}
Note that, since $\ind_{\Omega^{j,k}} \to 
\ind_{\tilde{\Omega}^j}$ strongly in $L^1$, 
up to extraction of a subsequence we can assume the existence
of $\tau\in (\rho/2, \rho)$ such that
\begin{equation}\label{e:Fubini}
\lim_{k\to\infty}
\|\ind_{\tilde{\Omega}^j} - \ind_{\Omega^{j,k}}\|_{L^1
(\partial B_\tau (y))}\;=\; 0\, .
\end{equation}
We also recall that, by the semicontinuity
of the perimeter,
\begin{equation}\label{e:semicont}
\per (\tilde{\Omega}^j)
\;\leq\; \liminf_{k\to\infty}
\cH^n (\partial \Omega^{j,k})\, .
\end{equation}
Define therefore the set $\Xi^{j,k}$ by setting
$$
\Xi^{j,k}\;=\; (\Xi \cap B_\tau (y))
\cup (\Omega^{j,k}\setminus B_\tau (y))\, .
$$
\eqref{e:competitor}, \eqref{e:Fubini}
and \eqref{e:semicont} imply
\begin{equation}\label{e:competitor2}
\limsup_{k\to\infty} [\per (\Xi^{j,k})
- \cH^n (\partial \Omega^{j,k})]
\;\leq\; -\eta\, .
\end{equation}
Fix next $k$ and
recall the following standard way of approximating $\Xi^{j,k}$
with a smooth set. We first fix a compactly supported
convolution kernel $\varphi$, then we consider the
function $g_\eps := \ind_{\Xi^{j,k}}*\varphi_\eps$
and finally look at a smooth level set $\Delta_\eps
:= \{g_\eps > t\}$ for some $t\in (\frac{1}{4},\frac{3}{4})$. 
Then $\cH^n (\partial\Delta_\eps)$ converges to $\per (\Xi^{j,k})$ as $\eps\to 0$
(see \cite{Giu} in the euclidean case and \cite{MP} for
the general one).

Clearly, $\Delta_\eps$ does not coincide anymore 
with $\Omega^{j,k}$ outside $B_\rho (y)$. 
Therefore, fix $(a,b)\subset (\tau, \rho)$ with the
property that $\Sigma := \Omega^{j,k}\cap \overline{B}_b (y)
\setminus B_a (y)$ is smooth. Fix
a regular tubular neighborhood $T$ of $\Sigma$ and corresponding
normal coordinates $(\xi,\sigma)$ on it.
Since
$\Xi^{j,k} \setminus B_\tau (y)
= \Omega^{j,k}\setminus B_\tau (y)$, for $\eps$ sufficiently small
$\partial \Delta_\eps\cap \overline{B}_b (y)
\setminus B_a (y)\subset T$ and $T\cap \Delta_\eps$
is the set $\{\sigma<f_\eps (\xi)\}$ for some smooth function
$f_\eps$. Moreover, as $\eps\to 0$,
$f_\eps \to 0$ smoothly. 

Therefore, a patching
argument entirely analogous to the one of the freezing
construction (see Subsection \ref{ss:freezing}), allows us to modify $\Xi^{j,k}$
to a set $\Delta^{j,k}$ with the following properties:
\begin{itemize}
\item $\partial \Delta^{j,k}$ is smooth outside of
a finite set;
\item $\Delta^{j,k}\setminus B= \Omega^{j,k}\setminus B$;
\item $\limsup_k (\cH^n (\partial\Delta^{j,k}) - \cH^n 
(\partial \Omega^{j,k}))\leq -\eta<0$.
\end{itemize} 
For $k$ large enough, Lemma \ref{l:concon}
implies that $\Xi^{j,k}\in \cH(\Omega^j, An)$,
which would contradict the minimality of the sequence
$\Omega^{j,k}$.

\medskip

Next, in order to show that the varifold $V^j$ is induced
by $\partial \tilde{\Omega}^j$, it suffices to show
that in fact $\cH^n (\partial \Omega^{j,k})$
converges to $\cH^n (\partial \tilde{\Omega}^j)$
(since we have not been able to find a precise
reference for this well--known fact, we give
a proof in the appendix; cp. with Proposition \ref{p:varivscacc}).
On the other hand, if this is not the case, then we
have
$$
\cH^n (\partial \tilde{\Omega}^j \cap B_{\rho/2} (y))
\;<\; \limsup_{k\to \infty} 
\cH^n (\partial \Omega^{j,k}\cap B_{\rho/2} (y))
$$
for some $y\in An$ and some $\rho$ to which we can
apply the conclusion Lemma \ref{l:concon}.
We can then use $\tilde{\Omega}^j$ in place
of $\Xi$ in the argument of the previous step
to contradict, once again, the minimality of the
sequence $\{\Omega^{j,k}\}_k$.
The stationarity and stability of the surface
$\partial \tilde{\Omega}^j$ is, finally, an obvious
consequence of the variational principle.

\begin{figure}[htbp]
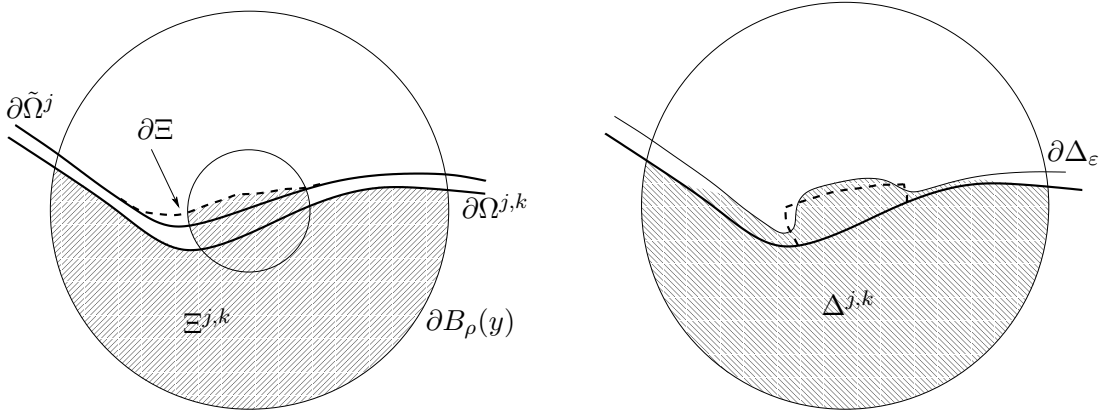

\begin{center}
\input{Pic4.pstex_t} 
\input{Pic4.1.pstex_t}
    \caption{On the left,
the set $\tilde{\Omega}^j$, the competitor $\Xi$,
one set of the sequence $\{\Omega^{j,k}\}_k$ and
the corresponding $\Xi^{j,k}$. On the right,
the smoothing $\Delta_\eps$ of $\Xi^{j,k}$ and the
final set $\Delta^{j,k}$ (a competitor for
$\Omega^{j,k}$).} 
    \label{f:competitor}
\end{center}
\end{figure}

\subsection{Proof of proposition \ref{p:replacement}}
\label{ss:rep_proof}

Consider the varifolds $V^j$ and
the diagonal sequence $\tilde{\Gamma}^j = \Gamma^{j,k(j)}$
of Section \ref{ss:rep_setting}. Observe
that $\tilde{\Gamma}^j$ is obtained from
$\Gamma^j$ through a suitable homotopy which
leaves everything fixed outside $An$.
Consider $An (x, \eps, r(x)-\eps)$ containing $An$. 
It follows from the a.m. property of 
$\{\Gamma^j\}$ that $\{\tilde{\Gamma}^j\}$ is also
a.m. in
$An (x,\eps, r(x)-\eps)$.

Note next that if a sequence is a.m. in an open
set $U$ and $U'$ is a second open set contained
in $U$, then the sequence is a.m. in $U'$ as well.
This trivial observation and the discussion above
implies that $\tilde{\Gamma}^j$ is a.m. in any
$An\in \An_{r(x)} (x)$.

Fix now an annulus $An'=An (x, \eps, r(x)-\eps)\supset\supset
An$. Then $M = An'\cup (M\setminus An)$.
For any $y\in M\setminus An$ (and $y\neq x$) 
consider $r'(y):=\min\{r (y), \dist (y, An)\}$. If
$An''\in \An_{r'(y)} (y)$, then $\Gamma^j\cap An'' =
\tilde{\Gamma}^j \cap An''$, and hence
$\{\tilde{\Gamma}^j\}$ is a.m. in $An''$. If
$y\in An'$, then we can set $r'(y) = \min\{r(y), \dist (y,
\partial An')\}$. If $An''\in \An_{r'(y)} (y)$,
then $An''\subset An'$ and, since $\{\tilde{\Gamma}^j\}$ is a.m.
in $An'$ by the argument above, $\{\tilde{\Gamma}^j\}$ is a.m.
in $An''$. 

We next show that $\tilde{V}$ is a replacement
for $V$ in $An$. By Theorem \ref{t:SScomp}, $\tilde{V}$
is a stable minimal hypersurface in $An$. 
It remains to show that $\tilde{V}$ is stationary.
$\tilde{V}$ is obviously stationary in $M\setminus An$,
because it coincides with $V$ there. Let next
$An'\supset\supset An$. Since $\{An', M\setminus An\}$
is a covering of $M$, we can subordinate a partition
of unity $\{\varphi_1, \varphi_2\}$ to it. By the linearity
of the first variation, we get
$[\delta \tilde{V}] (\chi) =[\delta \tilde{V}] (\varphi_1 \chi)
+ [\delta \tilde{V}] (\varphi_2\chi) = [\delta \tilde{V}] (\varphi_1 \chi)$.
Therefore it suffices to show that $\tilde{V}$ is stationary
in $An'$. Assume, by contradiction,
that there is $\chi\in \mathcal{X}_c (An')$ such that $[\delta \tilde{V}]
(\chi) \leq -C <0$ and denote by $\psi$ the isotopy defined 
by $\frac{\partial\psi(x,t)}{\partial t}=\chi(\psi(x,t))$. We
set
\begin{equation}
 \tilde{V} (t)\;:=\;\psi(t)_{\sharp} \tilde{V}\qquad
\Sigma^j (t)\;=\;\psi(t,\tilde{\Gamma}^j).
\end{equation}
By continuity of the first variation there is $\eps>0$ such that
$\delta \tilde{V}(t)(\chi)\leq-C/2$ for all $t\leq \eps$.
Moreover, since $\Sigma^j(t)\to \tilde{V}(t)$
in the sense of varifolds, there is $J$ such that 
\begin{equation}\label{e:deriv}
[\delta \Sigma^j (t)] (\chi)\leq 
-\frac{C}{4}\quad\text{for}\;j>J\;\text{and}\; t\leq\eps.
\end{equation}
Integrating \eqref{e:deriv} we conclude
$ \cH^n(\Sigma^j (t))\leq \cH^n (\tilde{\Gamma}^j) - Ct/8$
for every $t\in [0,\eps]$ and $j\geq J$. This
 contradicts the a.m. property
of $\tilde{\Gamma}^j$ in $An'$, for $j$ large enough.

Finally, observe that $\cH^n (\tilde{\Gamma}^j)
\leq \cH^n (\Gamma^j)$ by construction and
$\liminf_n (\cH^n (\tilde{\Gamma}^j)-\cH^n (\Gamma^j))\geq 0$,
because otherwise we would contradict the a.m. property
of $\{\Gamma^j\}$ in $An$. We thus conclude that $\|V\| (M)=
\|\tilde{V}\| (M)$.
\end{proof}

\section{The regularity of varifolds with replacements}
\label{s:reg}

In this section we prove Proposition \ref{p:reg}. We
recall that we adopt the convention of Definition
\ref{d:convention}.
We first list several technical facts from geometric measure theory.

\subsection{Maximum principle} The first one is just
a version of the classical maximum principle. 
 
\begin{theorem}\label{t:max_prin}
(i) Let $V$ be a stationary varifold in a ball $\mathcal{B}_r (0)\subset \R^{n+1}$.
If $\supp (V)\subset \{z_{n+1}\geq 0\}$ and $\supp (V)\cap \{z_{n+1}=0\}\neq\emptyset$,
then $\mathcal{B}_r (0)\cap \{z_{n+1}=0\}\subset \supp (V)$.

(ii) Let $W$ be a stationary varifold in an open set $U\subset M$ and
$K$ be a smooth strictly convex closed set. If $x\in \supp (V)\cap \partial K$,
then $\supp (V)\cap B_r (x)\setminus K \neq \emptyset$ for every positive $r$.
\end{theorem}

For (ii) we refer, for instance,
to Appendix B of \cite{CD}, whereas (i) is a very special case
of the general result of \cite{SW}. 

\subsection{Tangent cones}\label{ss:TC} The second device 
is a fundamental tool 
of geometric measure theory. Consider a stationary varifold
$V\in \mathcal{V} (U)$ with $U\subset M$ and fix a point $x\in \supp (V)\cap U$.
For any $r< \Inj (M)$ consider the rescaled exponential map
$T^x_r : \mathcal{B}_1 \ni z \mapsto \exp_x (rz)\in B_r (x)$,
where $\exp_x$ denotes the exponential map with base point $x$. We then
denote by $V_{x,r}$ the varifold $(T_r^x)^{-1}_\sharp V \in 
\mathcal{V} (\mathcal{B}_1)$. Then, as a consequence of the monotonicity
formula, one concludes that for any sequence $\{V_{x,r_n}\}$ there
exists a subsequence converging to a stationary varifold $V^*$
(stationary for the euclidean metric!), which in addition is a cone 
(see Corollary 42.6 of \cite{Si}). Any such cone
is called {\em tangent cone to $V$ in $x$}.
For varifolds with the replacement property, the following is a fundamental
step towards the regularity (first proved by Pitts for $n\leq 5$ in \cite{P}).

\begin{lemma}\label{l:tangent cones}
Let $V$ be a stationary varifold in an open set $U\subset M$ having
a replacement in any annulus $An\in \An_{r(x)}(x)$ for some positive function $r$.
Then:
\begin{itemize}
\item $V$ is integer rectifiable; 
\item $\theta(x,V)\geq 1$ for any $x\in U$;
\item Any tangent cone $C$ to $V$ at $x$ is a minimal hypersurface
for general $n$ and (a multiple of) a hyperplane for $n\leq 6$.
\end{itemize}
\end{lemma}
\begin{proof}
First of all, by the monotonicity formula there is a constant $C_M$ such that 
\begin{equation}\label{e:monotonicity}
\frac{\|V\|(B_\sigma(x))}{\sigma^n}\leq C_M\frac{\|V\|(B_\rho(x))}{\rho^n}
\qquad\mbox{for all $x\in M$ and all $0<\sigma\leq\rho<\Inj(M)$.}
\end{equation}
Fix $x\in \supp\,(\|V\|)$ and $0<r<\min \{r(x), \Inj (M)/4\}$.
Next, we replace $V$ with $V'$ in the annulus $An(x,r,2r)$.
We observe that $\|V'\|\not\equiv 0$ on 
$An(x,r,2r)$, otherwise there would be $\rho\leq r$ and $\eps$ such that 
$\supp\,(\|V'\|)\cap \partial B_\rho(x)\neq \emptyset$ and 
$\supp \,(\|V'\|)\cap \An(x, \rho,\rho+\eps)=\emptyset$. 
By the choice of $\rho$, this would contradict Theorem \ref{t:max_prin}(ii).

Thus we have found that $V'\res An(x, r,2r)$ is a non-empty stable
minimal hypersurface and hence 
there is $y\in An(x,r,2r)$ with $\theta(y,V')\geq 1$. By \eqref{e:monotonicity},
\begin{equation}\label{lower bound}
 \frac{\|V\|(B_{4r}(x))}{(4r)^n}=\frac{\|V'\|(B_{4r}(x))}{(4r)^n}\geq
\frac{\|V'\|(B_{2r}(y))}{(4r)^n}\geq\frac{\omega_n}{2^nC_M}\theta(y,V')\geq \frac{\omega_n}{2^nC_M}.
\end{equation}
Hence, $\theta(x,V)$ is uniformly bounded away from $0$ on $\supp\,(\|V\|)$
and Allard's Rectifiability Theorem (see Theorem 42.4 of \cite{Si}) gives that $V$ is rectifiable.

Let $C$ denote a tangent cone to $V$ at $x$ and $\rho_k\to 0$ a sequence with 
$V^x_{\rho_k}\to C$. Note that $C$ is stationary. We replace $V$ by $V'_k$ in 
$An(x,\lambda \rho_k,(1-\lambda) \rho_k)$, where $\lambda \in (0, 1/4)$
and set $W'_k=(T^x_{\rho_k})_\sharp V'_k$. 
Up to subsequences we have $W'_k\to C'$ for some stationary 
varifold $C'$. By the definition of a replacement we obtain
\begin{eqnarray}
 C'&=&C\quad\text{in}\;\cB_{\lambda}\cup An(0,1-\lambda,1),\label{e:coincide}\\
\|C'\|(\cB_\rho)&=&\|C\|(\cB_\rho)\quad\text{for}\;\rho\in(0,\lambda)\cup(1-\lambda,1).
\label{cone equality}
\end{eqnarray}
Moreover, since $C$ is cone, 
\begin{equation}\label{e:force_cone}
\frac{\|C'\|(\cB_\sigma)}{\sigma^n}=\frac{\|C'\|(\cB_\rho)}{\rho^n}\quad
\text{for all}\;\rho,\sigma\in (0,\lambda)\cup(1-\lambda,1).
\end{equation}
By the monotonicity formula for stationary varifolds in euclidean spaces,
\eqref{e:force_cone} implies
that $C'$ as well is a cone (see for instance 17.5 of \cite{Si}). 
Moreover, by the Compactness Theorem \ref{t:SScomp}, 
$C'\res An (0, \lambda, 1-\lambda)$
is a stable embedded minimal hypersurface. Since $C$ and $C'$ are
integer rectifiable, the conical structure of $C$ implies
that $\supp (C)$ and $\supp (C')$ are closed cones (in the usual meaning for sets)
and the densities $\theta (\cdot, C)$ and $\theta (\cdot, C')$ are 
$0$--homogeneous functions (see Theorem 19.3 of \cite{Si}). 
Thus \eqref{e:coincide} implies $C=C'$ and hence that $C$ is a stable
minimal hypersurface in $An (0, \lambda, 1-\lambda)$. Since
$\lambda$ is arbitrary, $C$ is a stable minimal
hypersurface in the punctured ball. Thus,
if $n\leq 6$, by Simons' Theorem (see Theorem B.2 in \cite{Si}) $C$ is in
fact a multiple of a hyperplane. If instead $n\geq 7$, since $\{0\}$ has dimension
$0\leq n-7$, $C$ is a minimal hypersurface in 
the whole ball $\mathcal{B}_1$ (recall Definition \ref{d:convention}).
\end{proof}

\subsection{Unique continuation and two technical lemmas
on varifolds} To conclude the proof we need yet three
auxiliary results. All of them are justified
in Appendix \ref{a:tech}.
The first one is a consequence of the
classical unique continuation for minimal surfaces.

\begin{theorem}\label{t:unique}
Let $U$ be a smooth open subset of $M$ and $\Sigma_1, \Sigma_2\subset U$ two
connected {\em smooth} embedded minimal hypersurfaces with $\partial \Sigma_i \subset 
\partial U$. If $\Sigma_1$ coincides with $\Sigma_2$ in some open subset of $U$,
then $\Sigma_1=\Sigma_2$. 
\end{theorem}

The other two are elementary lemmas
for stationary varifolds.

\begin{lemma}\label{l:tech1}
Let $r< \Inj (M)$ and $V$ a stationary varifold.
Then
\begin{equation}\label{e:tech1}
\supp (V)\cap \overline{B}_r (x) \;=\; 
\overline{\bigcup_{0<s<r} \supp (V\res B_s (x))\cap \partial B_s (x)}.
\end{equation}
\end{lemma}

\begin{lemma}\label{l:tech2}
Let $\Gamma\subset U$ be a relatively closed set of dimension $n$ and $S$ 
a closed set of dimension at most $n-2$ such that $\Gamma\setminus S$
is a smooth embedded hypersurface. Assume $\Gamma$ induces a varifold $V$
which is stationary in $U$. If $\Delta$ is a connected component of
$\Gamma\setminus S$, then $\Delta$ induces a stationary varifold.
\end{lemma}

\subsection{Proof of Proposition \ref{p:reg}}
{\bf Step 1: Set up.}
Let $x\in M$ and $\rho\leq \Inj (M)/2$. Then we choose a replacement $V'$ 
for $V$ in $An(x,\rho,2\rho)$
coinciding with a stable minimal embedded hypersurface $\Gamma'$. Next, choose
$s\in (0, \rho)$ and $t\in (\rho, 2\rho)$ such that $\partial B_t(x)$ intersects 
$\Gamma'$ transversally.
Then we pick a second replacement $V''$ of $V'$ in $An(x,s,t)$, coinciding 
with a stable minimal
embedded hypersurface $\Gamma''$ in the annulus $An(x,s,t)$. 
Now we fix a point $y\in \partial B_t(x)\cap \Gamma'$ that is a regular point of $\Gamma'$ and a radius $r>0$ sufficiently small such that $\Gamma'\cap B_r(y)$ is topologically an $n$-dimensional ball in $M$ and $\gamma=\Gamma'\cap \partial B_t(x)\cap B_r(y)$ is a smooth $(n-1)$-dimensional surface. This can be done due to our regularity assumption on $y$. Then we choose a diffeomorphism $\zeta: B_r(y)\to \cB_1$ such that
$$\zeta(\partial B_t(x))\subset\{z_1=0\}\quad\text{and}\quad \zeta(\Gamma'')\subset\{z_1>0\},$$
where $z_1,\dots,z_{n+1}$ are orthonormal coordinates in $\cB_1$. Finally suppose $$\zeta(\gamma)=\{(0,z_2,\dots,z_n,g'((0,z_2,\dots,z_n))\}\;\text{and}\; \zeta(\Gamma')\cap\{z_1\leq 0\}=\{(z_1,\dots,z_n,g'((z_1,\dots,z_n))\}$$
for some smooth function $g'$. Note that
\begin{itemize}
 \item any kind of estimates (like curvature estimates or area bound or monotonicity) for a minimal surface $\Gamma\subset B_r(y)$ translates into similar estimates for the surface $\zeta(\Gamma)$;
\item varifolds in $B_r(y)$ are pushed forward to varifolds in $\cB_1$ and there is a natural correspondence between tangent cones to $V$ in $\xi$ and tangent cones to $\zeta_{\sharp}V$ in $\zeta(\xi)$.
\end{itemize}
We will use the same notation for the objects in $B_r(y)$ and their images under $\zeta$.

\begin{figure}[htbp]
\begin{center}
\input{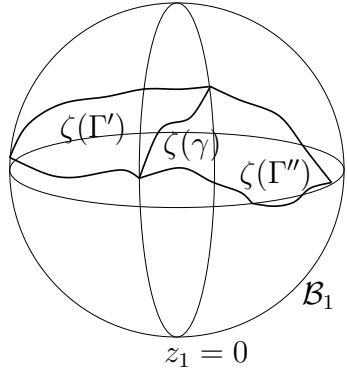} 
    \caption{The surfaces $\Gamma'$, $\Gamma''$ and $\gamma$
in the coordinates $z$.} 
    \label{f:setup}
\end{center}
\end{figure}

\textbf{Step 2: Tangent cones.}
We next claim that any tangent cone to $V''$ at any point $w\in \gamma$ is a unique flat space. 
Note that all these $w$ are regular points of $\Gamma'$. 
Therefore by our transversality assumption every tangent cone $C$ at $w$ coincides in 
$\{z_1<0\}$ with the half space $T_w\Gamma'\cap \{z_1<0\}$.
We wish to show that $C$ coincides with $T_w \Gamma'$. By the Constancy Theorem (see Theorem 41.1 in \cite{Si}), it suffices 
to show $\supp (C) \subset T_w \Gamma'$.

Note first that if $z\in T_w\Gamma'\cap \{z_1=0\}$ is a regular point for $C$,
then by Theorem \ref{t:unique}, $C$ coincides with $T_w \Gamma'$ in a neighborhood
of $z$. Therefore, if $z\in \supp (C)\cap
\{z_1=0\}$, either $z$ is a singular point, or $C=T_w \Gamma'$ in a neighborhood
of $z$. Assume now by contradiction that $p\in \supp (C)\setminus T_w \Gamma'$.
Since $\Sing\, C$ has dimension at most $n-7$, we can assume that $p$
is a regular point of $C$. Consider next a sequence $N^j$ of smooth open 
neighborhoods of $\Sing\, C$ such that $T_w \Gamma'\setminus \overline{N}^j$
is connected and $N^j\to \Sing\, C$. Let $\Delta^j$ be the connected
component of $C\setminus \overline{N}^j$ containing $p$. Then $\Delta^j$
is a smooth minimal surface with $\partial \Delta^j\subset \partial N^j$.
We conclude that $\Delta^j$ cannot touch
$\{z_1=0\}$: it would touch it in a regular point of $\supp (C)\cap \{z_1=0\}$
and hence it would coincide with $T_w \Gamma'\setminus \overline{N}^j$,
which is impossible because it contains $p$. If we let
$\Delta=\cup \Delta^j$, then $\Delta$ is a connected component of the regular
part of $C$, which does not intersect $\{z_1=0\}$. Let 
$W$ be the varifold induced by $\Delta$: by Lemma \ref{l:tech2} $W$
is stationary.
Since $C$ is a cone, $W$ is also a cone.
Thus $\supp (W)\ni 0$. On the other hand $\supp (W)\subset \{z_1\geq 0\}$.
Thus, by Theorem \ref{t:max_prin}(i), $\{z_1=0\}\subset \supp (W)$.
But this would imply that $\{z_1=0\}\cap T_w \Gamma'$ is in the singular
set of $C$: this is a contradiction because the dimension of $\{z_1=0\}\cap T_w \Gamma'$
is $n-1$.

\medskip

 \textbf{Step 3: Graphicality.}
In this step we show that the surfaces $\Gamma'$ and $\Gamma''$ 
can be ``glued'' together at $\partial B_t (x)$, that is
\begin{equation}\label{e:step3}
\mbox{$\Gamma''\subset \Gamma'$ in $B_t (x)\setminus B_{t-\eps} (x)$
for some $\eps>0$.}
\end{equation}
For this we fix $z\in \gamma$ and, using the notation of Step 2, 
consider the (exterior) unit normal 
$\tau(z)$ to the graph of $g'$. Let $T^z_r\colon\RR^{n+1}\to\RR^{n+1}$ be the dilation of the $(n+1)$-space given by
$$
T^z_r(\bar{z})=\frac{\bar{z}-z}{r}.
$$
By Step 2 we know that any tangent cone to $V''$ at $z$ is given by the 
tangent space $T_z\Gamma'$ and therefore the rescaled surfaces 
$\Gamma_r=T^z_r(\Gamma'')$ converge to the half space 
$H=\{v\colon \tau(z)\cdot v=0,v_1>0\}$. We claim that this implies that we have
\begin{equation}\label{normal convergence}
 \lim_{\bar{z}\to z,\bar{z}\in\Gamma''}
\frac{|(\bar{z}-z)\cdot \tau(z)|}{|\bar{z}-z|}=0
\end{equation}
uniformly on compact subsets of $\gamma$. We argue by contradiction and assume the claim is wrong. Then there is a sequence $\{z_j\}\subset \Gamma''$ with $z_j\to z$ and $|(z_j-z)\cdot \tau(z)|\geq k|z_j-z|$ for some $k>0$. We can assume that $z_j$ is a regular point of $\Gamma''$ for all $j\in \NN$. We set $r_j=|z_j-z|$, then there is a positive constant $\bar{k}$ such that $\cB_{2\bar{k}r_j}(z_j)\cap H=\emptyset$. This implies that $\dist(H, \cB_{\bar{k}r_j}(z_j))\geq \bar{k}r_j$. By the minimality of $\Gamma''$ we can apply the monotonicity formula and find
$$\|V''\|(\cB_{\bar{k}r_j}(z_j))\geq C\bar{k}^n r^n_j$$
for some positive constant $C$ depending on the diffeomorphism $\zeta$. In other words there is a considerable amount of the varifold that is far from the half space $H$. But this contradicts the fact that the corresponding full space is the only tangent cone. We also point out that this convergence is uniform on compact subsets of $\gamma$.\\
Now we denote by $\nu$ the smooth normal field to $\Gamma''$ with $\nu\cdot (0,\dots,0,1)\geq 0$. Let $\Sigma$ be the space $\{(0,\alpha_1,\dots,\alpha_n):\alpha_i\in \RR\}$. Then we assume that $z_j\to z$, set $r_j=\dist(z_j,\Sigma)$ and define the rescaled hypersurfaces $\Gamma_j=T^{z_j}_{r_j}(\Gamma''\cap \cB_{r_j}(z_j))$. 
Then all the $\Gamma_j$ are smooth stable minimal surfaces in $\cB_1$, 
thus we can apply Theorem \ref{t:SScomp} to extract a subsequence that 
converges to a stable minimal hypersurface in the ball $\cB_{1/2}$. 
But by \eqref{normal convergence} we know that this limit surface is simply 
$T_z\Gamma'\cap\cB_{1/2}$. 
Since the convergence is in the $C^1$ topology we have
$$
\lim_{\bar{z}\to z,\bar{z}\in \Gamma''}\nu(\bar{z})=\tau(z).
$$
Again this convergence is uniform in compact subsets of $\gamma$.

For any $z\in \gamma$ Theorem \ref{t:SScomp} gives us a radius $\sigma>0$ 
and a function $g''\in C^2(\{z_1\geq 0\})$ with
\begin{gather}
 \Gamma''\cap B_\sigma(z)=\{(z_1,\dots,z_n,g''(z_1,\dots,z_n)): z_1>0\}\\
g''(0,z_2,\dots, z_n)=g'(0,z_2,\dots,z_n)\quad\text{and}\quad Dg''
(0,z_2,\dots, z_n)=Dg'(0,z_2,\dots,z_n).
\end{gather}
Using elliptic regularity theory (see \cite{GT}), we conclude that 
$g'$ and $g''$ are the restriction of a smooth function $g$ 
giving a minimal surface $\Delta$.
Using now Theorem \ref{t:unique},
we conclude that $\Delta \subset \Gamma'$, and hence
that $\Gamma''$ is a subset of $\Gamma'$ in a neighborhood of $z$.
Since this is vaild for every $z\in \gamma$, we conclude
\eqref{e:step3}. 

\medskip

 \textbf{Step 4: Regularity in the annuli.}
In this step we show that $V$ is a minimal hypersurface in the punctured
ball $B_{\rho} (x)\setminus \{x\}$. First of all we prove
\begin{equation}\label{equalinannuli1}
\Gamma'\cap An(x,\rho,t)=\Gamma''\cap An(x,\rho,t). 
\end{equation}
Assume for instance that $p\in \Gamma''\setminus \Gamma'$. Without
loss of generality we can assume that $p$ is a regular point.
Let then $\Delta$ be the connected component of $\Gamma''\setminus
(\Sing\, \Gamma''\cup \Sing\, \Gamma')$ containing $p$. 
$\Delta$ is necessarily contained
in $\overline{B}_{t-\eps} (x)$, otherwise by \eqref{e:step3}
and Theorem \ref{t:unique}, $\Delta$ would coincide with
a connected component of $\Gamma'\setminus
(\Sing\, \Gamma''\cup \Sing\, \Gamma')$ contradicting $p\in \Gamma''\setminus \Gamma'$. But then $\Delta$ induces,
by Lemma \ref{l:tech2}, a stationary varifold $V$, with
$\supp (V)\subset \overline{B}_{t-\eps} (x)$. So, for some $s\leq t-\eps$,
we have $\partial B_s (x)\cap \supp (V)\neq\emptyset$
and $\supp (V)\subset \overline{B}_s (x)$, contradicting
Theorem \ref{t:max_prin}(ii). This proves $\Gamma''\subset \Gamma'$. Precisely the same argument
can be used to prove $\Gamma'\subset \Gamma''$. 

Thus we conclude that $\Gamma'\cup \Gamma''$ is in fact a 
minimal hypersurface in $An (x, s, 2\rho)$. Since $s$
is arbitrary, this means that $\Gamma'$ is in fact
contained in a larger minimal hypersurface $\Gamma
\subset B_{2\rho} (x)\setminus \{x\}$ and that, 
moreover, $\Gamma''\subset \Gamma$ for any
second replacement $V''$, whatever is the choice of $s$
($t$ being instead fixed).  

Fix now such a $V''$ and note that $V''\res B_s (x) =
V\res B_s (x)$. Note, moreover, that by Theorem
\ref{t:max_prin}(ii) we necessarily conclude
$$
\supp (V\res B_s (x))\cap \partial B_s (x)
\;\subset\; \overline{\Gamma''}\;\subset\;\Gamma\, .
$$
Thus, using Lemma \ref{l:tech1}, we conclude $\supp (V)\subset
\Gamma$, which hence proves the desired regularity of
$V$. 

\medskip

 \textbf{Step 5: Conclusion.}
The only thing left to analyize are the centers of the balls $B_\rho (x)$ 
of the previous steps. 
Clearly, if $n\geq 7$, we are done because by the compactness of $M$ we 
only have to add 
possibly a finite set of points, that is a 
$0$-dimensional set, to the singular set. In 
other words, the centers of the balls 
can be absorbed in the singular set.

If, on the other hand, $n\leq 6$, we need to show that $x$ is a regular point. 
If $x\notin \supp\,(\|V\|)$ we are done, so we assume $x\in \supp\, (\|V\|)$. 
By Lemma 
\ref{l:tangent cones} we know that every tangent cone is a multiple 
$\theta (x, V)$ of a plane (note that $n\leq 6$).
Consider the rescaled exponential maps of Section \ref{ss:TC}
and note that the rescaled varifolds
$V_r$ coincide with $(T^x_r)^{-1} (\Gamma)=\Gamma_r$. 
 Using Theorem \ref{t:SScomp}
we get the $C^1$--convergence of subsequences in $\cB_1\setminus\cB_{1/2}$
and hence the integrality of $\theta (x, V) = N$.
 
Fix geodesic coordinates in a ball $B_\rho (x)$. Thus, 
given any small positive constant $c_0$, if $K\in\N$ is sufficiently
large, there is a hyperplane $\pi_K$ such that,
on $An (x, 2^{-K-2}, 2^{-K})$, the varifold $V$ is 
the union of $m (K)$ disjoint graphs of Lipschitz functions
over the plane $\pi_K$, all
with Lipschitz constants smaller than $c_0$, counted with multiplicity
$j_1 (K), \ldots, j_m (K)$, with $j_1+\ldots + j_m=N$.
We do not know a-priori that there is a {\em unique} tangent 
cone to $V$ at $x$.
However, if $K$ is sufficiently large, it follows that
the tilt between two consecutive planes $\pi_K$ and $\pi_{K+1}$
is small. Hence $j_i (K)= j_i (K+1)$ and the corresponding Lipschitz graphs
do join, forming $m$ disjoint smooth minimal surfaces in the annulus
$An (x, 2^{-K-3}, 2^{-K})$, topologically equivalent to $n$--dimensional
annuli. Repeating the process inductively,
we find that $V\res B_\rho (x)\setminus \{x\}$ is in fact the union
of $m$ smooth disjoint minimal hypersurfaces $\Gamma^1, \ldots,
\Gamma^m$ (counted with multiplicities $j_1+\ldots + j_m=N$), which are
all, topologically, 
punctured $n$--dimensional balls. 

Since $n\geq 2$, by
Lemma \ref{l:tech2}, each $\Gamma^i$ induces a stationary varifold.
Every tangent cone to $\Gamma^i$ at $x$ is a hyperplane and,
moreover, the density of $\Gamma^i$ (as a varifold) is everywhere equal
to $1$. We can therefore apply Allard's regularity Theorem
(see \cite{All}) to conclude that each $\Gamma^i$ is regular.
On the other hand, the $\Gamma^i$ are disjoint in $B_r (x)\setminus
\{x\}$ and they contain $x$. Therefore, if $m>1$, we
contradict the classical maximum principle. We conclude that $m=1$
and hence that $x$ is a regular point for $V$.

\appendix

\section{Proofs of the technical lemmas}\label{a:tech}

\subsection{Varifolds and Caccioppoli set limits}
\begin{propos}\label{p:varivscacc}
 Let $\{\Omega^k\}$ be a sequence of Caccioppoli sets and 
$U$ an open subset of $M$. Assume that 
\begin{itemize}
 \item [(i)] $D \ind_{\Omega^k}\to D\ind_\Omega$ in 
the sense of measures in $U$;
 \item [(ii)] $\per (\Omega^k, U)\to\per(\Omega,U)$
\end{itemize}
for some Caccioppoli set $\Omega$ and denote by $V^k$ and $V$ the varifolds induced by $\partial ^\ast\Omega^k$ and $\partial^\ast\Omega$. Then $V^k\to V$ in the sense of varifolds.
\end{propos}
\begin{proof}
First, we note that by the rectifiability of the boundaries we can write
\begin{equation}\label{e:var_rep}
 V^k\;=\;\cH^n\res\partial^\ast\Omega^k\otimes
\delta_{T_x\partial^\ast\Omega^k}\qquad\mbox{and}
 \qquad
V\;=\;\cH^n\res\partial^\ast\Omega\otimes\delta_{T_x\partial^\ast\Omega}\, ,
\end{equation}
where $\partial^\ast\Omega,\partial^\ast\Omega^k$ are the reduced boundaries
and $T_x \partial^\ast \Omega$ is the approximate
tangent plane to $\Omega$ in $x$ (see Chapter 3 of \cite{Giu} for
the relevant definitions). With the notation $\mu \otimes \alpha_x$
we understand, as usual, the measure $\nu$ on a product space
$X\times Y$ given by
$$
\nu (E) \;=\; \int\int \ind_E (x,y)\, d\alpha_x (y)\,
d\mu (x)\, ,
$$
where $\mu$ is a Radon measure on $X$ and $x\mapsto \alpha_x$
is a weak$^\ast$ $\mu$--measurable map from $X$ into $\mathcal{M} (Y)$
(the space of Radon measures on $Y$). 

By $(ii)$ we have $\|V^k\|\to\|V\|$ and hence there is $W\in \mathcal{V}(U)$ such that (up to subsequences) $V^k\to W$. In addition, $\|V\|=\|W\|$. By the disintegration theorem (see Theorem 2.28 in \cite{AFP}) we can write 
$W=\cH^n\res\partial^\ast\Omega\otimes\alpha_x$. 
The proposition is proved, once we have proved
\begin{itemize}
\item[(Cl)] $\alpha_{x_0}=\delta_{T_{x_0}\partial^\ast\Omega}$ for 
$\cH^n$-a.e. $x_0\in \partial^\ast\Omega$.
\end{itemize}
To prove this, we reduce the situation to the case where 
$\Omega$ is a half space by a classical blow-up 
analysis. Having fixed a point $x_0$, a radius $r$, and
the rescaled exponential maps $T^r_{x_0}: \mathcal{B}_1 \to B_r (x_0)$
as in Subsection \ref{ss:TC}, we define
\begin{itemize}
\item $V^k_r := (T^r_{x_0})^{-1}_\sharp V^k$ and $V_r:= (T^r_{x_0})^{-1}_\sharp V$;
\item $\Omega^k_r := (T^r_{x_0})^{-1}(\Omega^k)$ and $\Omega_r:= (T^r_{x_0})^{-1} (\Omega)$.
\end{itemize}
Clearly, $V^k_r$ and $\Omega^k_r$ are related by
the same formulas as in \eqref{e:var_rep}. Next, let
$G$ be the set of radii $r$ such that $\cH^n (\partial^\ast
\Omega^k \cap \partial B_r (x_0))=\cH^n (\partial^\ast \Omega\cap \partial B_r (x_0))=0$
for every $k$ and observe that the complement 
of $G$ is a countable set.
Denote by $H$ the set
$\{x_1< 0\}$.
Then, after a suitable choice of orthonormal
coordinates in $\mathcal{B}_1$, we have 
\begin{itemize}
\item[(a)] $D\ind_{\Omega^k_r}\to D\ind_{\Omega_r}$
and $\per (\Omega^k_r, \mathcal{B}_1) \to
\per (\Omega_r, \mathcal{B}_1)$ for $k\to \infty$ and
$r\in G$;
\item[(b)] $D\ind_{\Omega_r} \to D\ind_H$ and
$\per (\Omega_r, \mathcal{B}_1)\to \per (H, \mathcal{B}_1)$
for $r\to 0, r\in G$;
\item[(c)] $T_0 \partial^* H = T_{x_0} \partial^\ast \Omega$; 
\item[(d)] $V^k_r \to V_r$ for $k\to \infty$ and $r\in G$. 
\end{itemize}
(The assumption $r\in G$ is essential: see Proposition
1.62 of \cite{AFP} or Proposition 2.7 of \cite{D}).

Next, for $\cH^n$--a.e. $x_0\in \partial^\ast \Omega$ we have in addition
\begin{itemize}
\item[(e)] $V_r \to \cH^n \res \partial^\ast H \otimes \alpha_{x_0}$
\end{itemize}
(in fact, if $\mathcal{D}\subset C (\mathbb{P}^n\RR)$ is a dense set,
the claim holds for every $x_0$ which is a point of approximate
continuity for all the functions $x\mapsto \int \varphi (y)
d\alpha_x (y)$ with $\varphi\in \mathcal{D}$).

By a diagonal argument we get sets $\tilde{\Omega}^k =
\Omega^k_{r(k)}$ such that
\begin{itemize}
\item[(f)] $D\ind_{\tilde{\Omega}^k} \to D \ind_H$ and
$\per (\tilde{\Omega}^k, \mathcal{B}_1) \to \per (H,
\mathcal{B}_1)$;
\item[(g)] $\cH^n \res \partial^\ast \tilde{\Omega}^k
\otimes \delta_{T_x \partial^* \tilde{\Omega}^k}
\to \cH^n \res \partial^\ast H \otimes \alpha_{x_0}$.
\end{itemize}
Let $e_1= (1,0,\ldots 0)$ and $\nu$ be the exterior unit normal 
to $\partial^\ast \Omega^k$. Then (f) implies
$$
\lim_{k\to \infty} \int_{\partial^\ast\tilde{\Omega}_k}\|\nu-e_1\|^2=
\lim_{k\to \infty} \left(2\cH^n(\partial^\ast\tilde{\Omega}_k)-
2\int_{\partial^\ast\tilde{\Omega}_k}\langle\nu,e_1\rangle\right)
\;=\; 0\, .
$$
This obviously gives $\cH^n \res \partial^\ast \tilde{\Omega}^k
\otimes \delta_{T_x \partial^* \tilde{\Omega}^k} \to
\cH^n \res \partial^\ast H \otimes \delta_{T_0 \partial^* H}$,
which together with (c) and (g) gives $\alpha_{x_0}
= \delta_{T_0 \partial^\ast H} = \delta_{T_{x_0} \partial^\ast
\Omega}$, which is indeed the Claim (Cl).   
\end{proof}

\subsection{Proof of Theorem \ref{t:unique}}
Let $W\subset U$ be the maximal open set on which $\Sigma_1$ and
$\Sigma_2$ coincide. If $W\neq U$, then there is a point $p\in 
\overline{W}\cap U$. In a ball $B_\rho (p)$, $\Sigma_2$ is
the graph of a smooth function $w$ over $\Sigma_1$ (as usual,
we use normal coordinates in a regular neighborhood of $\Sigma_1$).
By a straightfoward computation, $w$ satisfies a 
differential inequality of the form $| A^{ij} D^2_{ij} w|\leq C (|Dw|+|w|)$
where $A$ is a smooth function with values in symmetric matrices,
satisfying the usual ellipticity condition $A^{ij}\xi_i\xi_j \geq \lambda
|\xi^2|$, where $\lambda>0$. Let $x\in U$ be such that
$\dist (x,p)< \eps$. Then $w$ vanishes at infinite order in $x$
and hence, according to the classical result
of Aronszajn (see \cite{Ar}), $w\equiv 0$ on a ball $B_r (x)$ where
$r$ depends on $\lambda$, $A$, $C$ and $\dist (x, \partial B_\rho (p))$, 
but not on $\eps$. Hence,
by choosing $\eps<r$ we contradict the maximality of $W$.

\subsection{Proof of Lemma \ref{l:tech1}} 
Let $T$ be the set of points $y\in \supp (V)$ such that the approximate
tangent plane to $V$ in $y$ is transversal to the sphere $\partial B_{|y-x|}
(x)$. The claim follows from the density of $T$ in $\supp (V)$. 
The (quite short) proof of this statement can be found for instance
in Appendix B of \cite{CD} (cp. with Lemma B.2 therein).

\subsection{Proof of Lemma \ref{l:tech2}}
Set $\Gamma_r:= \Gamma\setminus S$ and denote by $H$ the 
mean curvature of $\Gamma_r$ and by $\nu$ the unit normal
to $\Gamma_r$. Obviously $H=0$.
Let $V'$ be the varifold induced by $\Delta$.
We claim that
\begin{equation}\label{e:MC}
[\delta V'] (\chi) \;=\; \int_{\Delta} {\rm div}_{\Delta}\,
\chi \;=\;- \int_{\Delta} H \chi\cdot \nu\, 
\end{equation}
for any vector field $\chi\in \mathcal{X}_c (U)$. 

The first identity is the classical computation of the first
variation (see Lemma 9.6 of \cite{Si}). To prove the second identity,
fix a vector field $\chi$ and a constant $\eps>0$.
W.l.o.g. we assume $S\subset \Gamma$.
By the definition of the Hausdorff measure, there exists 
a covering of $S$ with balls $B_{r_i} (x_i)$ centered
on $x_i\in S$ such that
$r_i<\eps$ and $\sum_i r_i^{n-1}\leq \eps$. By the compactness
of $S\cap \supp (\chi)$ we can find a finite covering
$\{B_{r_i} (x_i)\}_{i\in \{1, \ldots, N\}}$. Fix smooth cutoff
functions $\varphi_i$ with 
\begin{itemize}
\item $\varphi_i=1$ on $M\setminus B_{2r_i} (x_i)$
and $\varphi_i=0$ on $B_{r_i} (x_i)$;
\item $0\leq \varphi_i\leq 1$, $|\nabla \varphi_i|\leq C r_i^{-1}$.
\end{itemize}
(Note that $C$ is in fact only a geometric constant.)
Then $\chi_\eps \;:=\; \chi\Pi \varphi_i$ 
is compactly supported in $U\setminus S$. Thus,
\begin{equation}\label{e:byparts}
\int_\Delta {\rm div}_\Delta\,
\chi_\eps \;=\; -\int_\Delta H \chi_\eps\cdot \nu\, 
\end{equation}
The RHS of \eqref{e:byparts} obviously converges to the RHS
of \eqref{e:MC} as $\eps\to 0$. As for the left hand 
side, we estimate
\begin{eqnarray}
&&\int_\Delta \left|{\rm div}_\Delta (\chi-\chi_\eps)\right|
\;\leq\; \sum_i \int_{B_{r_i} (x_i)\cap \Delta} 
(\|\nabla \chi\|_{C^0} + \|\chi\|_{C^0}
\|\nabla \varphi_i\|_{C^0})\nonumber\\
&\leq& \sum_i \|V\| (B_{r_i} (x_i)) \|\chi\|_{C^1}
(1+ C r_i^{-1}) \;\leq\; C \|\chi\|_{C^1} \sum_i (r_i^n+C r_i^{n-1}) \;<\;
C\eps\label{e:estHaus}
\end{eqnarray}
where the second inequality in the last line follows from the monotonicity
formula. We thus conclude that the LHS of \eqref{e:byparts} converges
to the LHS of \eqref{e:MC}.

\nocite{*}
\bibliographystyle{plain}
\bibliography{bibliografia}

\end{document}